\documentclass{article}
 
\usepackage[latin1]{inputenc}
\usepackage[T1]{fontenc}
\usepackage[francais, english]{babel}
\usepackage{lmodern}
\usepackage{stmaryrd}
\usepackage{amsthm}
\usepackage{amsmath}
\usepackage{amssymb}
\usepackage{mathrsfs}
\usepackage{soul}
\usepackage{layout}
\usepackage{setspace}
\usepackage[top=3cm, bottom=3cm, left=3cm, right=3cm]{geometry}
\usepackage{pgf,tikz}
\usepackage{graphicx}
\usepackage{xcolor}
\usetikzlibrary{arrows}
\usepackage{listings}
\usepackage{enumitem}
\usepackage{hyperref}
\usepackage{cases}
\usepackage{mathabx}
\usepackage{cleveref}
\usepackage{sectsty}
\usepackage{fancyhdr}

\newtheorem{th.}{Theorem}[section]
\newtheorem{lemme}[th.]{Lemma}
\newtheorem{fact}{Fact}
\newtheorem*{lemme*}{Lemme} 
\newtheorem{prop.}[th.]{Proposition}
\newtheorem{def.}{Definition}[section]
\newtheorem{cor.}[th.]{Corollary}
\newtheorem*{cor*}{Corollary}
\newtheorem*{rem.}{Remarque}
\newtheorem*{ex.}{Exemple}
\newtheorem*{th.*}{Theorem}
\newtheorem*{th-traj-geod}{\Cref{th-traj-geod} bis}
\newtheorem*{prop-trans}{\Cref{th-trans}}
\newtheorem*{Th-rec/trans}{\Cref{th-rec/trans}}
\newtheorem*{distform2bis}{\Cref{distform2} bis}
\newtheorem{theorem}{Theorem}

\newcommand{\D}{\mathbb{D}}
\newcommand{\Z}{\mathbb{Z}}

\newcommand{\C}{\mathcal{C}}
\newcommand{\R}{\mathbb{R}}

\newcommand{\N}{\mathbb{N}}
\newcommand{\PP}{\mathscr{P}}
\newcommand{\FF}{\mathscr{F}}

\newcommand{\ka}{\mathfrak{a}}

\newcommand{\leb}{\text{leb}}

\subsectionfont{\fontsize{13}{17}\selectfont}

\subsubsectionfont{\fontsize{12}{15}\selectfont}

\makeatletter
\newcounter {subsubsubsection}[subsubsection]
\renewcommand\thesubsubsubsection{\thesubsubsection .\@alph\c@subsubsubsection}
\newcommand\subsubsubsection{\@startsection{subsubsubsection}{4}{\z@}%
                                     {-3.25ex\@plus -1ex \@minus -.2ex}%
                                     {1.5ex \@plus .2ex}%
                                     {\normalfont\large\bfseries}}
\newcommand*\l@subsubsubsection{\@dottedtocline{3}{10.0em}{4.1em}}
\newcommand*{\subsubsubsectionmark}[1]{}
\makeatother

\title{Some asymptotic properties of random walks on homogeneous spaces}
\author{Timothée Bénard}
\date{}

\begin{document}

\maketitle

\bigskip

\abstract{Let $G$ be a connected semisimple real Lie group with finite center, and $\mu$ a  probability measure on $G$ whose support generates a Zariski-dense subgroup of $G$. We consider the right $\mu$-random walk on $G$ and show that each random trajectory spends most of its time  at bounded distance of  a well-chosen Weyl chamber. We infer that if $G$ has rank one, and $\mu$ has a finite first moment, then for any discrete subgroup $\Lambda\subseteq G$, the $\mu$-walk and the geodesic flow on $\Lambda \backslash G$ are either both transient, or both recurrent and ergodic, thus extending a well known theorem due to  Hopf-Tsuji-Sullivan-Kaimanovich dealing with the Brownian motion.}

\bigskip

\bigskip

\large
\tableofcontents

\bigskip

\bigskip

\bigskip

{\small{\bf  Mathematics Subject Classification }:  22E40, 37B20, 37H15.} 

{\small {\bf Key words}  :  Lie groups, Homogeneous spaces, Random walks, Geodesic flow, Recurrence.} 

\newpage
\large
\section{Introduction}
\baselineskip=15pt

\bigskip

This paper studies the asymptotic properties of random walks on semisimple Lie groups or their quotients. This topic has been developed for 60 years. The heart of the subject  is the theory of random walks on linear groups  \cite{BouLac, BQRW} worked out by Furstenberg \cite{Fur63}, Kesten \cite{Kes}, Guivarc'h \cite{GuiRau86}, and many others, to transpose classical limit theorems for Markov chains on $\Z^d$ to the context of linear random walks.  It recently led to spectacular applications to  walks on finite volume homogeneous spaces, such as  Eskin-Margulis Theorem \cite{EskMar}  establishing the non escape of mass for any starting point, or later  Benoist-Quint's classification  of stationary probability measures \cite{BQI, BQII, BQIII} (see also \cite{EskLin}). Our paper adds to this network of ideas by enriching the general theory of walks on linear groups and infering a recurrence criterion for walks in infinite volume.

\bigskip
A seminal result due to Furstenberg \cite{Fur63} is that a trajectory of a Zariski-dense random walk on a semisimple Lie group goes to infinity in a specific direction, given by a point on the flag variety.  In the very concrete setting of a walk on $PSL_{2}(\R)$, seen as the unitary bundle of the Poincaré disk $\mathbb{D}$, this means that every walk trajectory converges toward a limit point on the boundary $\partial\D$. Furstenberg's result has since been transposed to the context of Gromov hyperbolic spaces. More precisely, Ancona  \cite{Anc} considers the Brownian motion on a Gromov hyperbolic manifold and shows that the distance between a Brownian trajectory and its limit geodesic ray grows at most logarithmically. Analogous results for walks on  hyperbolic groups are proven in  \cite{Anc88,Led, BHM,Sis17} . 

The first theorem of our paper completes this  panel of results by claiming that the distance between a random trajectory and its asymptotic geodesic ray (or more generally asymptotic Weyl chamber) is most of the time bounded. We state and prove our result in the context of walks on Lie groups, eventhough the method  could be adapted to deal with hyperbolic groups.

 Let $G$ be a connected semisimple real Lie group with finite center,  $\mu$ a probability measure on $G$.  The (right) $\mu$-random walk on $G$ is defined by the transition probabilities 
$$p(g,h)=\mu(g^{-1}h) $$
A  trajectory starting from a point $x_{0}\in G$ is thus obtained as a sequence $(x_{0}b_{1}\dots b_{n})_{n\geq 1}$ where the $b_{i}\in G$ are independent $\mu$-distributed increments.  We make the assumption that the subgroup $\Gamma_{\mu}=\langle \text{supp } \mu \rangle \subseteq G$ generated by the support of $\mu$ is \emph{Zariski-dense} in $G$. Denoting by $\mathfrak{g}$ the Lie algebra of $G$, this means that every polynomial function on $\text{End}(\mathfrak{g})$ which vanishes on the adjoint representation $\text{Ad}\Gamma_{\mu}$ is also null on $\text{Ad}G$.

\smallskip

\newpage

 \begin{theorem}[Bounded deviations]\label{th-traj-geod}
 Let $G$ be a connected semisimple  real Lie group with finite center, and $\mu$ a probability measure on $G$ such that  $\Gamma_{\mu}$  is Zariski-dense in $G$. Set $B=G^{\N^\star}$, $\beta=\mu^{\otimes \N^\star}$, fix  a maximal compact subgroup $K\subseteq G$, and a left invariant metric $d$ on $G$. 
 
 To every $b\in B$, one can associate a Weyl Chamber $\C(b)\subseteq G/K$ such that the following holds. For all $\varepsilon>0$, there exists a constant $R>0$, such that for $\beta$-almost every $b\in B$,
 $$\liminf_{n\to +\infty} \frac{1}{n}\sharp \{i\in \llbracket 1, n\rrbracket,\,\, d(b_{1}\dots b_{i}, \C(b))\leq R \} >1-\varepsilon $$

  \end{theorem}
 \bigskip
\smallskip
Here  $d(b_{1}\dots b_{i}, \C(b))$ refers to  the distance between $b_{1}\dots b_{i}$ and  $\C(b)$ seen as a right $K$-invariant subset of $G$. 

Recall that when $G$ has rank $1$ (e.g. $G=SO_{e}(d,1)$ or $G=SU(d,1)$), then a Weyl chamber  of $G/K$ is just a geodesic ray for the symmetric space structure of $G/K$. In higher rank, it corresponds to a convex cone in a maximal flat of $G/K$. The result we prove is actually slightly more precise than \Cref{th-traj-geod}:  the map $b\mapsto \C(b)$ is  explicit  in terms of the Cartan decompositions of $(b_{1}\dots b_{n})_{n\geq 1}$, and we  bound the distance between $b_{1}\dots b_{n}$ and a particular point in $\C(b)$  (see \Cref{ThAbis}).

  In concrete linear algebra terms, for  $G=SL_{d}(\R)$, $K=SO_{d}(\R)$, we can set $\ka^+=\{t=\text{diag}(t_{1}, \dots, t_{d}), \,t_{1}\geq \dots\geq t_{d},\, \sum t_{i}=0\}$ and choose $\C(b)$ of the form $\C(b)=k_{\infty}(b)\exp(\ka^+)K$ where $k_{\infty}(b)\in K$.  Writing $b_{1}\dots b_{n}K=k_{n}(b)\exp(t_{n}(b))K$  where $t_{n}(b)\in \ka^+$ is the so-called Cartan projection, our deviation result  \emph{bounds the  difference of angle $k^{-1}_{n}(b)k_{\infty}(b)$ in terms of $t_{n}(b)$}:  for all $\varepsilon>0$,  for $\beta$-typical $b\in B$, there is a subset $S_{\varepsilon, b}\subseteq \N^*$ of asymptotic density at least $(1-\varepsilon)$ such that 
$$\,\,\,\,\,\,\,\,\,\,\,\,(k^{-1}_{n}(b)k_{\infty}(b))_{i,j} =O_{\varepsilon}(\exp(t_{n}(b)_{i} -t_{n}(b)_{j})) \,\,\,\,\,\,\,\,\,\,\,\,\,\,\,\,\,\,\,\,\,\text{ for $n\in S_{\varepsilon, b}$, $i>j$}$$
   In particular $k^{-1}_{n}(b)k_{\infty}(b)$ converges to $\text{Id}$ (up to sign of coefficients) at  speed $O_{\varepsilon}\left(\exp(-\min_{i\neq j} |t_{n}(b)_{i} -t_{n}(b)_{j}|)\right)$ along $n \in S_{\varepsilon, b}$. More information on the asymptotic behavior of $t_{n}(b)$ can be found in \cite{BQRW} (in particular \cite[Th. 10.9]{BQRW}).
   

Finally, we emphasize that \emph{no moment assumption} is made on $\mu$ in \Cref{th-traj-geod}. In contrast, the logarithmic bounds of \cite{Anc88, Led, BHM, Sis17} all rely strongly on the hypothesis of a finite exponential moment (or at least the H\"older regularity of the harmonic measure). 
Note also that \Cref{th-traj-geod} is already new in the case where the $\mu$-walk is a (discrete) Brownian motion.

 \bigskip

\bigskip

Our second theorem considers the $\mu$-random walk induced on a quotient $X=\Lambda\backslash G$, where $\Lambda$ is a discrete subgroup of $G$, and characterizes the situations of recurrence or transience in terms of the geodesic flow. A first result of this kind was obtained by Tsuji \cite{Tsu59}, who built on Hopf's alternative \cite{Hopf39, Hopf71} to prove that the Brownian motion and the geodesic flow on a hyperbolic surface are either both recurrent ergodic, or both transient. Sullivan \cite{Sul} extended Hopf-Tsuji Theorem to  hyperbolic manifolds of arbitrary dimension, and Kaimanovich \cite{Kai} pushed it even further, dealing with walks on rank one symmetric spaces. However,   all these theorems concern a Brownian motion  or at least a spread-out random walk, i.e. a $\mu$-walk such that $\mu$ (or a convolution power) is absolutely continuous with respect to the Haar measure on $G$. In this paper, we extend them to walks determined  by an arbitrary Zariski-dense probability measure $\mu$ with a finite first moment, meaning that
$$\int_{G}\log||\text{Ad}g||\,d\mu(g)<\infty $$

\bigskip

Let us now prepare our statement  by recalling some basic definitions.  More details can be found in \cite[1.2-1.3]{Tim-these}.

Assume $G$ has rank one and let $\mathfrak{a}\subseteq \mathfrak{g}$ be a Cartan subspace of $G$ which is orthogonal to the Lie algebra of $K$ for the Killing form. For $t\in\mathfrak{a}$, set $a_{t}=\exp(t)\in G$. The one-parameter subgroup $(a_{t})_{t\in\mathfrak{a}}$ acts by right multiplication on $X$,  inducing a flow that we call the \emph{geodesic flow}. This terminology is natural as any geodesic path  on the locally symmetric space $X/K$  is of the form $t\mapsto xa_{t}K$ for a suitable $x\in X$ \cite[Proposition 4.3]{Paul-riem}.

Let $F\subseteq X$ be a compact set of positive Haar measure. The \emph{Green function} of the $\mu$-walk associated to $F$, denoted by $G_{\mu}(.,F) : X\rightarrow [0,+\infty]$,  estimates the average time spent in $F$ by the $\mu$-trajectories starting at a given point.   We may define in a similar way a Green function for the geodesic flow  $G(.,F) : X\rightarrow [0,+\infty]$. The precise formulas are 
$$G_{\mu}(x,F)=\int_{B}\sum_{n\geq 0}1_{F}(xb_{1}\dots b_{n})\, d\beta(b)\,\,\,\,\,\,\,\,\,\,\,\,\,\,\,G(x,F) = \int_{\mathfrak{a}^+}1_{F}(xa_{t}) \,dt$$
where $\mathfrak{a}^+\subseteq \mathfrak{a}$ is a fixed Weyl chamber of $\mathfrak{a}$.

The $\mu$-walk on $X$ is  \emph{recurrent} (resp. \emph{transient}) if for almost-every $x\in X$, $\beta$-almost every $b\in B$, the trajectory $(xb_{1}\dots b_{n})_{n\geq 0}$ sub-converges to $x$ (resp. leaves every compact). It is equivalent to say that for every $F$, almost-every $x\in X$, one has $G_{\mu}(x,F) \in \{0, +\infty\}$ (resp. $G_{\mu}(x, F)<+\infty$). 

The $\mu$-walk on $X$ is \emph{ergodic} if the subgroup $\Gamma_{\mu}$ of $G$ generated by the support of $\mu$ acts ergodically on $X$ for the Haar measure. In the context of a recurrent random walk, this amounts to say  that $G_{\mu}(x,F)=\infty$ for every $F$ and almost-every $x\in X$. 

Analogous definitions of recurrence, transience, ergodicity hold for the geodesic flow on $X$.

\newpage
\smallskip

\begin{theorem}[Recurrence criterion] \label{th-rec/trans}
Let $G$ be a  connected simple real Lie group of rank one, $\Lambda\subseteq G$ a discrete subgroup, $X=\Lambda\backslash G$. Let $\mu$ be a probability measure on $G$ with a finite first moment and with $\Gamma_{\mu}$ Zariski-dense in $G$. 

 Then the $\mu$-walk and  the geodesic flow  on $X$ are either both recurrent  ergodic, or both transient with  locally integrable Green functions. 
\end{theorem}

\smallskip
A striking consequence is that the recurrence properties of a walk on a rank-one homogeneous space do not depend on the generating measure $\mu$ but only on the geometry of the  space. Also note that in the particular case where $G=Sp(1,m)$ for some $m\geq 2$, or $G=F^{-20}_{4}$,  the geodesic flow (or equivalently the $\mu$-walk) is always transient as long as $\Lambda$ has infinite covolume in $G$ (see \Cref{Kaimth} and the subsequent remark).

\bigskip

\noindent{\bf Organization of the paper.} 

\Cref{Sec-traj-geod} is dedicated to \Cref{th-traj-geod}. It is the occasion  to set the notations,  recall the dynamics of Zariski-dense random walks on the flag variety of $G$, and introduce a parametrization of the set of maximal flats of $G/K$ which will  also be useful in the rest of the paper.

 \Cref{rec-crit} is dedicated to \Cref{th-rec/trans}. The proof that the transience of the geodesic flow on $X$ implies the transience of the walks is entirely contained in the section. The converse relies on the framework that ermerged in  \Cref{Sec-traj-geod}, more specifically \Cref{distform}. Both statements use renewal results for the Cartan projection on rank one simple real Lie groups.
 
\Cref{Hopfdich}  is an appendix that recalls the Hopf alternative for the geodesic flow, usually stated in terms of Poincaré series, and explains how it can be formulated via the Green functions of the flow.

\bigskip
\noindent{\bf Acknowledgments.} The author is grateful to Yves Benoist and to the anonymous referee for their helpful comments.

\bigskip

\section{Random trajectories and  asymptotic Weyl chambers} \label{Sec-traj-geod}

This section is dedicated to the proof \Cref{th-traj-geod}.

\subsection{Notations and strategy} \label{sec-not/strat}

Throughout \Cref{Sec-traj-geod}, we denote by  $G$  a  connected non-compact semisimple  real Lie group with finite center, $K\subseteq G$ a maximal compact subgroup, $\mathfrak{g}, \mathfrak{k}$ their respective Lie algebras, $\mathfrak{a}\subseteq \mathfrak{g}$ a Cartan subspace orthogonal to $\mathfrak{k}$ for the Killing form, $\mathfrak{a}^+\subseteq \mathfrak{a}$ a closed Weyl chamber of $\mathfrak{a}$, and for $t\in \mathfrak{a}$ we set $a_{t}=\exp(t)\in G$.   Recall that in this context, every element $g\in G$ can be written as $g=k_{g}a_{t_{g}}l_{g}$ where $k_{g}, l_{g}\in K$, $t_{g}\in \mathfrak{a}^+$. This is called the Cartan decomposition, and the element $t_{g}$ is unique and called the Cartan projection of $g$. 

Let $\mu$ be a probability measure on $G$ and $\Gamma_{\mu}\subseteq G$ the subgroup generated by its support. We suppose that $\Gamma_{\mu}$ is Zariski-dense in $G$, but make no moment assumption on $\mu$. We denote by $(B, \beta)=(G^{\N^\star}, \mu^{\otimes \N^\star })$ the space of instructions guiding the $\mu$-walk on $G$.

\bigskip
 
{\bf The proof of \Cref{th-traj-geod} can be summarized as follows}. We fix for every $b\in B$, $n\geq 0$ a Cartan decomposition 
$$b_{1}\dots b_{n}= k_{n}(b)a_{t_{n}(b)}l_{n}(b) $$
 \Cref{dynamicsP} guarantees that we can choose those decompositions such that for $\beta$-almost every $b\in B$, the sequence $(k_{n}(b))\in K^{\N^\star}$ has a limit $k_{\infty}(b)\in K$. We aim to control the distance $d(b_{1}\dots b_{n}, k_{\infty}(b)a_{t_{n}(b)})$ given a left $G$-invariant riemmanian metric on $G$. To this end, we define the flag varieties $G/P^-$, $G/P$, denote by $\xi^-_{0} =P^-/P^-$, $\xi_{0}=P/P$ their basepoints, and set $\FF^+=G.(\xi^-_{0}, \xi_{0})$ the $G$-orbit of $(\xi^-_{0}, \xi_{0})$ in $G/P^-\times G/P$. \Cref{param-flat} introduces a map $(\xi^-, \xi)\mapsto F(\xi^-,\xi)$ from $\FF^+$ to the space of maximal flats in $G/K$. It is a $G$-equivariant cover such that $F(\xi^-_{0},\xi_{0})=\exp(\mathfrak{a})K$.  We see in \Cref{sec-distform}  that for some constant $C_{0}>0$, every $\xi^-\in G/P^-$, $\beta$-almost every $b\in B$, large enough $n\geq 0$, 
$$d(b_{1}\dots b_{n}, k_{\infty}(b)a_{t_{n}(b)})\leq C_{0}d(b_{1}\dots b_{n}, F(\xi^-, \xi_{b}))+C_{0} $$
where $\xi_{b}=k_{\infty}(b)\xi_{0}$ is the limit point of $(b_{1}\dots b_{n})$ on the  flag variety $G/P$. Finally, in \Cref{sec-proof-traj-geod}, we use Birkhoff Ergodic Theorem to control the right-hand side of the previous inequality and conclude the proof of \Cref{th-traj-geod}. 
 
\bigskip
\noindent{\bf Remark.} The method of proof is inspired by Ledrappier's paper \cite{Led}  establishing logarithmic deviation on free groups between sample paths  and their asymptotic geodesic rays, under a finite exponential moment condition. As pointed out by the referee, it is also related to \cite{Tio15} in which Tiozzo proves sublinear deviation of sample paths under very general assumptions.

\subsection{Dynamics of walks on the flag variety} \label{dynamicsP}

We recall here basic facts about the dynamics of the $\mu$-walk on the flag variety of $G$. The existence of the limit direction $k_{\infty}$ introduced in \ref{sec-not/strat} is also justified.

\bigskip

Let us begin with the definition of the flag variety. For $\alpha \in  \mathfrak{a}^\star$, set 
$$\mathfrak{g}_{\alpha}:=\{s\in \mathfrak{g},\, \forall t \in \mathfrak{a}, \,\, (\text{ad}\,t)(s)=\alpha(t)s \}$$
 As the action of $(\text{ad}\,t)_{t\in \mathfrak{a}}$ on $\mathfrak{g}$ is simultaneously diagonalizable, we can write 
$$\mathfrak{g}= \mathfrak{g}_{0} \oplus \bigoplus_{\alpha \in \Phi}\mathfrak{g}_{\alpha} $$ 
  where $\Phi :=\{\alpha \in \mathfrak{a}^\star\smallsetminus\{0\},\, \mathfrak{g}_{\alpha}\neq \{0\} \}$ is the root system of $\mathfrak{a}$. 
   Denote by $\Phi^+\subseteq \Phi$ the subset of positive roots given by $\mathfrak{a}^+$, namely $\Phi^+=\{\alpha\in \Phi,\, \alpha(\mathfrak{a}^+)\subseteq \R^{+}\}$. Set  $\mathfrak{u}=\bigoplus_{\alpha \in \Phi^+}\mathfrak{g}_{\alpha}$
and let $P=N_{G}(\mathfrak{g}_{0}\oplus\mathfrak{u})$ be the  subgroup of elements $g\in G$ whose adjoint action preserves the subspace $\mathfrak{g}_{0}\oplus\mathfrak{u}$. The \emph{flag variety} of $G$ is defined as the quotient 
$\mathscr{P}=G/P$. We set $\xi_{0}=P/P$ the standard base point.

\bigskip

It will often be convenient to work with a concrete realisation of $\PP$ as some $G$-orbit in a product of projective spaces. To this end, we recall the following fact \cite[Lemma 6.32]{BQRW}. Let $\Pi$ be the basis of $\Phi$ prescribed by $\Phi^+$.

\begin{fact} \label{Fact1}
There exists a family $(V_{\alpha}, \rho_{\alpha})_{\alpha\in \Pi}$  of proximal  irreducible algebraic representations of $G$ such that 
\begin{itemize}
\item denoting by $\xi_{\alpha}\in \mathbb{P}(V_{\alpha})$ the line of highest weight of $(V_{\alpha}, \rho_{\alpha})$, we have a $G$-equivariant embedding 
$$\PP\hookrightarrow \prod_{\alpha\in \Pi} \mathbb{P}(V_{\alpha}), \,\,g\xi_{0}\mapsto (\rho_{\alpha}(g)\xi_{\alpha})_{\alpha\in \Pi}  $$
\item the heighest weights $(\chi_{\alpha})_{\alpha\in \Pi}$ form a basis of $\mathfrak{a}^\star$. Moreover $\chi_{\alpha}-\alpha$ is also a weight of $(V_{\alpha}, \rho_{\alpha})$.
\end{itemize}
\end{fact}

We equip each $V_{\alpha}$ with a scalar product  that is $\rho_{\alpha}(K)$-invariant  and such that every element of $T\rho_{\alpha}(\mathfrak{a})$ is self-adjoint \cite[Lemma 6.33]{BQRW}. 
\bigskip

We know by \cite[Proposition 10.1]{BQRW} that $\PP$ admits a unique stationary probability measure, we call it $\nu_{\PP}$. This measure gives no mass to proper subvarieties, and is proximal: for $\beta$-almost every $b\in B$, the limit measure $\nu_{\PP, b}=\lim (b_{1}\dots b_{n})_{\star}\nu_{\PP}$ is a Dirac mass $\nu_{\PP,b}=\delta_{\xi_{b}}$.  We then have $\nu_{\PP}=\int_{B}\delta_{\xi_{b}} d\beta(b)$. The $\alpha$-coordinate map $\psi_{\alpha}:\PP\rightarrow \mathbb{P}(V_{\alpha}), g\xi_{0}\mapsto \rho_{\alpha}(g)\xi_{\alpha}$ sends $\nu_{\PP}$ to the unique $\mu$-stationary probability measure $\nu_{\mathbb{P}(V_{\alpha})}$ on  $\mathbb{P}(V_{\alpha})$. It is atom free and proximal, with limit measures $\nu_{\mathbb{P}(V_{\alpha}),b}=\delta_{\xi_{b, \alpha}}$ where $\xi_{b, \alpha}=\psi_{\alpha}(\xi_{b})$.



\bigskip
Recall that we have fixed for every $b\in B$, $n\geq0$, a  Cartan decomposition $b_{1}\dots b_{n}=k_{n}(b)a_{t_{n}(b)}l_{n}(b)$. We show in \Cref{lem-F1}  that the sequence $(k_{n}(b)\xi_{0})_{n\geq0}$ converges in $\PP$. As $\PP=K/M$ where $M=Z_{K}(\mathfrak{a})$, this result ensures we can always choose our decompositions so that $(k_{n}(b))_{n\geq0}$ converges in $K$,  justifying  the definition of $k_{\infty}$ given in \Cref{sec-not/strat}.

\begin{lemme}\label{lem-F1}
For $\beta$-almost every $b\in B$, one has the convergence in $\PP$
$$k_{n}(b)\xi_{0}\underset{n\to +\infty}{\longrightarrow} \xi_{b}  $$
\end{lemme}

\smallskip
\begin{proof} It is enough to argue for each coordinate, i.e. show that for each $\alpha\in \Pi$, 
$$\rho_{\alpha}(k_{n}(b))\xi_{\alpha}\underset{n\to +\infty}{\longrightarrow} \xi_{b, \alpha}  $$
To lighten the proof, we just write $g$ for $\rho_{\alpha}(g)$.  

Notice that
$$\frac{b_{1}\dots b_{n}}{||b_{1}\dots b_{n}||}= k_{n}(b) \frac{a_{t_{n}(b)} }{||a_{t_{n}(b)}||} l_{n}(b)$$
\cite[Proposition 4.7]{BQRW} states that any accumulation point in $\text{End}(V_{\alpha})$ of the sequence $(\frac{b_{1}\dots b_{n}}{||b_{1}\dots b_{n}||})_{n\geq 1}$ must be of rank one with image $\xi_{b, \alpha}$. This yields that  the sequence $\frac{a_{t_{n}(b)}}{||a_{t_{n}(b)}||}$ converges to the orthogonal projection on the line of heighest weight $\xi_{\alpha}$. 
Let $k'_{\infty}(b)$ be some limit value of the sequence $k_{n}(b)$, and $\sigma : \N \rightarrow \N$  an extraction such that $k_{\sigma(n)}(b)\rightarrow k'_{\infty}(b)$ and $l_{\sigma(n)}(b)$  converges in $K$. Then $\frac{b_{1}\dots b_{\sigma(n)}}{||b_{1}\dots b_{\sigma(n)}||}$  converges as well and its limit has image $k'_{\infty}(b)\xi_{\alpha}$. Hence $\xi_{b, \alpha} =k'_{\infty}(b)\xi_{\alpha}$, which proves the lemma.

\end{proof}

\subsection{Parametrization of maximal flats} \label{param-flat}

Let $\PP^-$ be the opposite flag variety of $G$ defined by setting $\mathfrak{u}^-=\bigoplus_{\alpha \in \Phi\smallsetminus \Phi^+}\mathfrak{g}_{\alpha}$, $P^-=N_{G}(\mathfrak{g}_{0}\oplus\mathfrak{u}^-)$ and $\PP^-=G/P^-$. Let  $\xi^-_{0}=P^-$, $\xi_{0}=P$ be the base points of $\PP^-$, $\PP$ and denote by $\FF^+=G.(\xi^-_{0},\xi_{0})$ the (open) $G$-orbit of $(\xi^-_{0},\xi_{0})$ in $\PP^-\times \PP$. The following lemma associates to every pair of flags $(\xi^-,\xi)\in \FF^+$ a maximal flat of $G/K$. 

\bigskip
\begin{lemme}

There exists a unique $G$-equivariant map 
$$\begin{array}{rcl}
&\FF^+ \,\longrightarrow\,& \{\text{maximal flats of $G/K$}\}\\
&(\xi^-,\xi) \mapsto &F(\xi^-,\xi)
\end{array}$$
such that $(\xi^-_{0},\xi_{0})\mapsto \exp(\mathfrak{a})K$. Moreover this map is a Galois cover whose group of deck transformations is the Weyl group of $G$. 
\end{lemme}

\bigskip
\noindent{\bf Remark}. The map $F$ is also used in \cite[Section 10]{Kai00} to describe the Poisson boundary of walks on discrete subgroups of semi-simple Lie groups, via the strip approximation method.
\smallskip
\begin{proof}

We can identify $\FF^+$ to  $G/(P^-\cap P)=G/Z_{G}(\mathfrak{a})$ and the set of maximal flats of $G/K$ to $G/N_{G}(\mathfrak{a})$. The map we are to define is just the quotient projection map $G/Z_{G}(\mathfrak{a})\rightarrow G/N_{G}(\mathfrak{a})$. Moreover the Weyl group $W=Z_{G}(\mathfrak{a})\backslash N_{G}(\mathfrak{a})$ is finite and acts freely on $G/Z_{G}(\mathfrak{a})$ by right multiplication,  so the above projection map is a Galois cover whose group of deck transformations is $W$. 

\end{proof}

We now check that for every $\xi^-\in \PP^-$, and $\beta$-almost every $b\in B$, the maximal flat $F(\xi^-,\xi_{b})$ is well defined. To this end, we begin with the following criterion 

 \begin{lemme} \label{crit}
 Let $g_{1},g_{2}\in G$. 
 $$(g_{1}\xi^-_{0}, g_{2}\xi_{0})\in \FF^+ \iff g^{-1}_{1}g_{2}\in P^-P$$
 
\end{lemme}
\smallskip

\begin{proof}
\begin{align*}
(g_{1}\xi^-_{0}, g_{2}\xi_{0}) \in \FF^+ &\iff (\xi^-_{0}, g^{-1}_{1} g_{2}\xi_{0}) \in G.(\xi^-_{0},\xi_{0})\\
&\iff \exists h\in G,\,h\xi^-_{0}=\xi^-_{0} \text{ and }  h\xi_{0}= g^{-1}_{1} g_{2}\xi_{0}\\
&\iff \exists h\in P^-,\, h^{-1} g^{-1}_{1} g_{2} \in P
\end{align*}

\end{proof}

 \begin{lemme}\label{lem-xib}
 For every $\xi^-\in \PP^-$, and $\beta$-almost every $b\in B$, one has $(\xi^-,\xi_{b})\in \FF^+$. 
 \end{lemme}

\smallskip
\begin{proof}
Write $\xi^-=g_{1}\xi^-_{0}$ for some $g_{1}\in G$. According to \Cref{crit}, we need to check that for $\beta$-almost every $b\in B$, 
$$\xi_{b}\in g_{1}P^-\xi_{0}$$
Bruhat decomposition \cite[Theorem 5.15]{BT65} guarantees that $\PP = \sqcup_{w\in W} g_{1}P^{-}w\xi_{0}$
where $W$ stands for the Weyl group.  It is thus enough to show that for $w\in W\smallsetminus\{0\}$, we have $\nu_{\PP}(g_{1}P^{-}w\xi_{0})=0$.

 As $w\neq 0$, we have $w\xi_{0}\neq \xi_{0}$, so by \Cref{Fact1}, there exists $\alpha \in\Pi$ such that $\rho_{\alpha}(w)\xi_{\alpha}\neq \xi_{\alpha}$.   As $\rho_{\alpha}(w)$ permutes the weights of $V_{\alpha}$, it has to send $\xi_{\alpha}$ to $V^<_{\alpha}$ the unique $\mathfrak{a}$-invariant complementary subspace of  $\xi_{\alpha}$. But $V^<_{\alpha}$ is stable under $\rho_{\alpha}(P^-)$. To sum up, we have 
 $$\psi_{\alpha}(g_{1}P^-w \xi_{0})\subseteq \mathbb{P}(\rho_{\alpha}(g_{1})V^<_{\alpha})$$ where $\psi_{\alpha}$ is the $\alpha$-coordinate projection $\PP\rightarrow \mathbb{P}(V_{\alpha}), g\xi_{0}\mapsto \rho_{\alpha}(g)\xi_{\alpha}$.  
As the action of $G$ on  $V_{\alpha}$ is irreducible, the stationary measure $\nu_{\mathbb{P}(V_{\alpha})}=\psi_{\alpha \star}\nu_{\PP}$   gives no mass to proper projective spaces \cite[Lemma 4.6]{BQRW}.  In particular, the above  yields $\nu_{\PP}(g_{1}P^-w \xi_{0})=0$.  Finally $\nu_{\PP}(g_{1}P^-\xi_{0})=1$.

\end{proof}

\subsection{Distance formula} \label{sec-distform}

The goal of \Cref{sec-distform} is to prove \Cref{distform} which bounds the distance from $b_{1}\dots b_{n}$ to $k_{\infty}(b)a_{t_{n}(b)}$ by the distance from $b_{1}\dots b_{n}$ to a well-chosen maximal flat of $G/K$.  \Cref{distform} will  follow from  geometric (non-random) considerations encapsulated in \Cref{distform0}.  
We endow $G$ with a left $G$-invariant Riemannian metric, and denote by $d$ the distance induced on $G$. 
For $\xi\in \PP$, we  write $\xi=k_{\xi}\xi_{0}$ where $k_{\xi}\in K$. Recall  that for $x\in G$, we denote by $x=k_{x}a_{t_{x}}l_{x}$ where $k_{x},l_{x}\in K$, $t_{x}\in \ka^+$ a Cartan decomposition of $x$. The Cartan projection $t_{x}$ is uniquely defined, and $k_{x}, k_{\xi}$ are uniquely defined in $K/M$ (where $M=Z_{K}(\ka)$) as long as $t_{x}$ is in the interior of $\ka^+$ . 

\begin{prop.} \label{distform0}
  There exists  $C_{0}>0$ such that for  all $(\xi^-, \xi)\in \FF^+$, all sequence $(x_{n})\in G^\N$ such that  $\inf_{\alpha\in \Phi^+}\alpha (t_{x_{n}})\to +\infty$ and $k_{x_{n}} \to k_{\xi}$ in $K/M$, we have for $n$ large enough, 
$$d(x_{n}, k_{\xi}a_{t_{x_{n}}})  \leq C_{0}d(x_{n}, F(\xi^-, \xi))+C_{0}$$
\end{prop.}

\begin{cor.} \label{distform}
  There exists   $C_{0}>0$ such that for  all $\xi^-\in \PP^-$,  for $\beta$-almost every $b\in B$, for large enough $n\geq 0$, 
$$d(b_{1}\dots b_{n}, k_{\infty}(b)a_{t_{n}(b)})  \leq C_{0}d(b_{1}\dots b_{n}, F(\xi^-, \xi_{b}))+C_{0}$$
\end{cor.}

\bigskip
We first prove \Cref{distform0}. It relies on the following technical lemma.

\begin{lemme}\label{distform2}
There exists a constant $C_{1}>0$ such that for every $u\in \exp(\mathfrak{u})$, there exists a neighborhood $V_{K}\subseteq K$ of the neutral element $e$ in $K$ such that for all $s,t\in \mathfrak{a}, k\in V_{K}$ 
$$d(a_{s},a_{t})\leq C_{1}d(ua_{s}, ka_{t})+C_{1} $$
\end{lemme}

\bigskip
Let us see first how to deduce \Cref{distform0} from  \Cref{distform2}.

\begin{proof}[Proof of  \Cref{distform0}]
Let $(\xi^-, \xi)\in \FF^+$ and $(x_{n})$ as in  \Cref{distform0}.  Using \Cref{param-flat}, $$F(\xi^-, \xi)= k_{\xi}F( k_{\xi}^{-1}\xi^-, \xi_{0})= k_{\xi}u_{\xi^-}\exp(\mathfrak{a})K$$
for some element $u_{\xi^-}\in \exp(\mathfrak{u})$. The assumption $\inf_{\alpha\in \Phi^+}\alpha (t_{x_{n}})\to +\infty$ implies that $a_{-t_{x_{n}}}u_{\xi^-}a_{t_{x_{n}}}\rightarrow 0$ as $n$ goes to infinity. Hence we can write for $n\geq 0$, $s\in \mathfrak{a}$, 
\begin{align*}
d(x_{n}, k_{\xi} a_{t_{x_{n}}})&=d(x_{n}, k_{\xi} u_{\xi^-}a_{t_{x_{n}}})+o(1)\\
&\leq d(x_{n}, k_{\xi} u_{\xi^-} a_{s})+ d(a_{s},a_{t_{x_{n}}})+o(1)
\end{align*}
Using the assumption that $k_{x_{n}}\rightarrow k_{\xi}$ in $K/M$ and \Cref{distform2}, we have for large $n\geq $0, every $s\in \mathfrak{a}$, 
\begin{align*}
d(a_{s},a_{t_{x_{n}}})&\leq C_{1}d(k_{\xi}^{-1}k_{x_{n}} a_{t_{x_{n}}},  u_{\xi^-}a_{s})+C_{1}\\
&= C_{1}d(k_{x_{n}}a_{t_{x_{n}}},  k_{\xi}u_{\xi^-} a_{s})+C_{1}\\
&\leq C_{1}d(x_{n},  k_{\xi}u_{\xi^-}a_{s})+C_{2}\\
\end{align*}
where $C_{2}=C_{1}\text{diam}K+C_{1}$, which leads to 
$$d(x_{n}, k_{\xi} a_{t_{x_{n}}}) \leq (1+C_{1})d(x_{n},  k_{\xi}u_{\xi^-}a_{s})+ C_{2}+o(1)$$
Choosing $s$ to realise the infimum, we obtain  for large enough $n\geq0$,
$$d(x_{n}, k_{\xi} a_{t_{x_{n}}}) \leq (1+C_{1})d(x_{n}, F(\xi^-, \xi))+C_{2}+1$$
which concludes the proof.
\end{proof}

\bigskip
We now need to show \Cref{distform2}. 

\begin{proof}[Proof of \Cref{distform2}]
Notice first that if $d_{1}, d_{2}$ are the distances induced on $G$ by two left $G$-invariant Riemannian metrics, then there exists a constant $R>0$ such that
$$\frac{1}{R}d_{2}\leq d_{1}\leq R\,d_{2} $$ 
Hence, in order to prove \Cref{distform2}, we can specify $d$ as follows.  Let $\mathfrak{s}=\mathfrak{k}^\perp$ be  the orthogonal of $\mathfrak{k}$ for  the Killing form $\mathcal{K}: \mathfrak{g}\times \mathfrak{g}\rightarrow \R$.  We know by \cite{BQRW} that $\mathcal{K}$ is negative definite on $\mathfrak{k}$ and positive definite on $\mathfrak{s}$. In particular we have a decomposition $ \mathfrak{g}=  \mathfrak{k}\oplus  \mathfrak{s}$ and we may define a saclar product on $\mathfrak{g}$ by setting $\langle .,.\rangle=-\mathcal{K}(\theta.,.)$ where $\theta=\text{Id}_{\mathfrak{k}}\oplus -\text{Id}_{\mathfrak{s}}$ is the opposition involution map. We endow $G$ with left $G$-invariant metric that coincides with $\langle .,.\rangle$ on $\mathfrak{g}$ and write $d$ the corresponding distance map. 

It is then a standard exercise to check that for $s\in \mathfrak{a}$, 
$$d(e, a_{s})=||s||$$
where $||.||$ is the euclidean norm associated to $\langle .,.\rangle$. In particular, for any $g,h\in G$, 
$$ ||\kappa(g^{-1}h)||-2 \text{diam} K\leq  d(g,h)\leq ||\kappa(g^{-1}h)||+2 \text{diam} K $$ 
where $\kappa : G\rightarrow \mathfrak{a}^+$ denotes the Cartan projection map. 

\bigskip
This inequality means that we may reformulate \Cref{distform1} as follows

\begin{distform2bis}\label{3}There exists a constant $C_{2}>0$ such that for every $u\in \exp(\mathfrak{u})$, there exists a neighborhood $V_{K}\subseteq K$ of the neutral element $e$ in $K$ such that for all $s,t\in \mathfrak{a}$, $k\in V_{K}$ 
$$||s-t|| \leq C_{2}||\kappa(a_{-s}uka_{t})||+C_{2}$$
\end{distform2bis}

\begin{proof}[Proof of Lemma \ref{distform2} bis]
We use the representations $(V_{\alpha}, \rho_{\alpha})_{\alpha\in \Pi}$ introduced in \Cref{Fact1}. For $\alpha\in \Pi$, let  $v_{\alpha}\in  \xi_{\alpha}$ be a vector in the line of heighest weight of $V_{\alpha}$ such that $||v_{\alpha}||=1$. Recall from \cite[Section 6.8]{BQRW} that $\rho_{\alpha}(u)(v_{\alpha})=v_{\alpha}$. In particular, there exists a neighborhood $V_{K,\alpha}\subseteq K$ of $e$ in $K$ such that for all $k\in V_{K,\alpha}$, we have $\rho_{\alpha}(uk)v_{\alpha}=v'_{\alpha}+v''_{\alpha}$ with $v'_{\alpha}\in \xi_{\alpha}$, $v''_{\alpha}\in \xi^\perp_{\alpha}$ and $||v'_{\alpha}||\geq \frac{1}{2}$. Let $s,t\in \mathfrak{a}$ and $k\in V_{K,\alpha}$.
$$||\rho_{\alpha}(a_{-s}uka_{t}v_{\alpha})||\geq \frac{1}{2}e^{\chi_{\alpha}(t-s)} $$ 
which leads to 
$$\chi_{\alpha}(t-s)\leq  \log  ||\rho_{\alpha}(a_{-s}uka_{t})||+  \log 2$$
However, $\log ||\rho_{\alpha}(a_{-s}uka_{t})||= \chi_{\alpha}(\kappa(a_{-s}uka_{t}))\leq ||\chi_{\alpha}||\, ||\kappa(a_{-s}uka_{t})|| $,
so $$\chi_{\alpha}(t-s)\leq   ||\chi_{\alpha}||\, ||\kappa(a_{-s}uka_{t})|| +\log 2$$
As $\kappa(g)=\kappa(g^{-1})$, we can apply the previous argument  to $(a_{-s}uka_{t})^{-1}$ to strengthen the previous inequality  and get for  $s,t\in \mathfrak{a}$, $k\in V'_{K,\alpha}$ neighborhood of $e$ in $K$, 
 $$|\chi_{\alpha}(t-s)|\leq ||\chi_{\alpha}||\, ||\kappa(a_{-s}uka_{t})|| + \log 2 $$
Now assuming $k\in V_{K}=\cap_{\alpha\in \Pi}V'_{K,\alpha}$ and summing over $\alpha\in \Pi$, 
 \begin{align} \label{eq2}
 \sum_{\alpha\in \Pi}|\chi_{\alpha}(t-s)|\leq (\sum_{\alpha\in \Pi}||\chi_{\alpha}||)\, ||\kappa(a_{-s}uka_{t})|| + \sharp \Pi \log 2 
 \end{align}
As the weights $(\chi_{\alpha})_{\alpha\in \Pi}$ form a basis of $\mathfrak{a}^\star$, there exists a constant $C>0$, depending only on  $(\chi_{\alpha})_{\alpha\in \Pi}$ and $||.||$, such that  
 \begin{align}\label{eq3}
 ||t-s||\leq C\sum_{\alpha\in \Pi}|\chi_{\alpha}(t-s)| 
  \end{align}
  Inequalities (\ref{eq2}) and  (\ref{eq3})  together prove Lemma \ref{distform2} bis. 
  
\end{proof}
\renewcommand{\qedsymbol}{}
\end{proof}

We now turn to the proof of \Cref{distform}. In order to apply \Cref{distform0}, we prove  

\begin{lemme}\label{distform1}
For $\beta$-almost every $b\in B$, for every $\alpha\in \Phi^+$, 
$$\alpha(t_{n}(b))\underset{n\to +\infty}{\longrightarrow} +\infty$$
\end{lemme}

\begin{proof}[Proof of \Cref{distform1}]
We only need to show that for every $\alpha\in \Pi$, $\beta$-almost every $b\in B$, 
$$\alpha(t_{n}(b))\underset{n\to +\infty}{\longrightarrow} +\infty$$
Consider again the representations $(V_{\alpha}, \rho_{\alpha})$ introduced in \Cref{Fact1}.  Arguing as in \Cref{lem-F1}, we see that $\frac{\rho_{\alpha}(a_{t_{n}(b)})}{||\rho_{\alpha}(a_{t_{n}(b)})||}$ converges to the orthogonal projection on the line of heighest weight $\xi_{\alpha}$ in $V_{\alpha}$.  In particular, given a vector $w_{\alpha}$ in  the weight space of $\chi_{\alpha}-\alpha$, we get $\frac{\rho_{\alpha}(a_{t_{n}(b)})}{||\rho_{\alpha}(a_{t_{n}(b)})||}(w_{\alpha})\rightarrow 0$.  Noticing that $ ||\rho_{\alpha}(a_{t_{n}(b)})||= e^{\chi_{\alpha}(t_{n}(b))}$, the latter can be rewritten as  
$$e^{-\alpha(t_{n}(b))} \underset{n\to +\infty}{\longrightarrow} 0$$ 
which concludes the proof.
\end{proof}

\bigskip
\begin{proof}[Proof of \Cref{distform}]
It follows from the combination of \Cref{distform0} and Lemmas \ref{lem-xib}, \ref{distform1}, \ref{lem-F1}.

\end{proof}

\subsection{Proof of  \Cref{th-traj-geod} } \label{sec-proof-traj-geod}

We  conclude \Cref{Sec-traj-geod} with the proof of \Cref{th-traj-geod}. We actually show the following more detailed version. 

\bigskip
 \begin{th-traj-geod} \label{ThAbis}  Keep the notations of \Cref{sec-not/strat}.  In particular, $G$ is a connected semisimple  real Lie group with finite center,  $\mu$  a probability measure on $G$ with  $\Gamma_{\mu}$ Zariski-dense in $G$,  set $B=G^{\N^\star}$, $\beta=\mu^{\otimes \N^\star}$ and for $\beta$-almost every $b\in B$,  choose a Cartan decomposition $b_{1}\dots b_{n}=k_{n}(b)a_{t_{n}(b)}l_{n}(b)$ with $k_{n}(b)$ converging in $K$, and let $k_{\infty}(b)=\lim k_{n}(b)$.
 
 Then for all $\varepsilon>0$, there exists a constant $R>0$ such that for $\beta$-almost every $b\in B$, 
$$\liminf_{n\to +\infty} \frac{1}{n}\sharp\{i\in \llbracket 1, n\rrbracket,\,\, d(b_{1}\dots b_{i}, k_{\infty}(b)a_{t_{i}(b)})\leq R\} >1-\varepsilon $$
 \end{th-traj-geod}

\bigskip 
\begin{proof}
In view of \Cref{distform}, it is enough to prove that for all $\varepsilon>0$, there exists a constant $R>0$ and an element $\xi^-\in \PP^-$, such that for  $\beta$-almost every $b\in B$, 
\begin{align*}
\liminf_{n\to +\infty} \frac{1}{n}\sharp\{i\in \llbracket 1, n\rrbracket,\,\, d(b_{1}\dots b_{i}, F(\xi^-, \xi_{b}))\leq R\} >1-\varepsilon 
\end{align*}
The  observation  that 
\begin{align}
d(b_{1}\dots b_{i}, F(\xi^-, \xi_{b}))=d(e, F(b^{-1}_{i}\dots b^{-1}_{1}\xi^-, \xi_{T^ib})) \label{eq4} 
\end{align}
 where $T : B \rightarrow B, b=(b_{i})_{i\geq 1}\mapsto (b_{i+1})_{i\geq 1}$ is the one-sided shift, motivates the following. 

 Let $\widecheck{\mu}$ be the image of $\mu$ under the inversion map $g\mapsto g^{-1}$, and  $\nu_{\PP^-}$  the $\widecheck{\mu}$-stationary probability measure on $\PP^-$. Define $T^+ :B\times \PP^-\rightarrow B\times \PP^-, (b,\xi^-)\mapsto (Tb, b^{-1}_{1}\xi^-)$.  The  ergodicity of the $\widecheck{\mu}$-walk on $\PP^-$ is a consequence of \cite[Proposition 4.7]{BQRW} and means that  the dynamical system $(B\times \PP^-, \beta\otimes \nu_{\PP^-}, T^+)$ is measure-preserving and ergodic \cite[Proposition 2.14]{BQRW} . Define (almost everywhere) a function $f :  B\times \PP^-\rightarrow [0,+\infty[$ setting
$$
f(b, \xi^-)  = \left\{
    \begin{array}{ll}
        1 & \mbox{if } d(e, F(\xi^-, \xi_{b})) \leq R \\
        0 & \mbox{otherwise}
    \end{array}
\right.
$$
Notice that $f$ is measurable\footnote{To check this, observe that the map $\phi : \FF^+\rightarrow \R^+, (g\xi^-_{0},g\xi_{0})\mapsto d(e, F(g\xi^-_{0},g\xi_{0}) )= d(g^{-1}, F(\xi^-_{0},\xi_{0}) )$ is continuous, hence its extension to $\PP^-\times \PP$ by setting $\phi=+\infty$ on the (closed) complement $\PP^-\times \PP\smallsetminus \FF^+$ is measurable. Now the measurability of $f$ follows from the measurability of $B\rightarrow \PP, b\mapsto \xi_{b}$.}
and that we may choose $R>0$ large enough so that $\beta\otimes \nu_{\PP^-}(f)>1-\varepsilon$. In this case,  Birkhoff Ergodic Theorem implies that for $\nu_{\PP^-}$-almost every $\xi^-\in \PP^-$, $\beta$-almost every $b\in B$, large enough $n\geq0$,
\begin{align*}
\frac{1}{n}\sum^{n}_{i=1}f \circ (T^+)^i(b, \xi^-)>1-\varepsilon \end{align*}
which can be rewritten as
\begin{align*}
 \frac{1}{n}\sharp\{i\in \llbracket 1, n\rrbracket,\,\, d(b_{1}\dots b_{i}, F(\xi^-, \xi_{b}))\leq R\} >1-\varepsilon 
\end{align*}
This concludes the proof. 
\end{proof}

\bigskip

\bigskip

\section{Recurrence criterion} \label{rec-crit}

The goal of this section is to prove  our second theorem announced in the introduction. 

\bigskip
\begin{Th-rec/trans}
Let $G$ be a  connected simple real Lie group of rank one, $\Lambda\subseteq G$ a discrete subgroup, $X=\Lambda\backslash G$. Let $\mu$ be a probability measure on $G$ with a finite first moment and $\Gamma_{\mu}$ Zariski-dense in $G$. 

 Then the $\mu$-walk and  the geodesic flow  on $X$ are either both recurrent  ergodic, or both transient with  locally integrable Green functions. 
\end{Th-rec/trans}

\bigskip
We will use freely the notations of \ref{sec-not/strat} and always be in the setting of \Cref{th-rec/trans}. In particular, $\mathfrak{a}$ denotes a Cartan subspace of dimension $1$, that we will identify with $\R$ via the linear isomorphism sending $1\in \R$ to the element $v_{0}\in \mathfrak{a}^+$  of norm $1$ for the Killing form.  In this regard, for $t\in \R$, we have by definition $a_{t}=\exp(t v_{0})\in G$.

As we shall explain in  \Cref{Hopfdich}, the dichotomy presented in  \Cref{th-rec/trans} is already known for the geodesic flow  :

\smallskip
\noindent{\bf Fact 2.}
\emph{The geodesic flow $(a_t)_{t\in \R}$ on $X$ is  either recurrent  ergodic, or  transient with  locally integrable Green functions.}

\smallskip

 Hence  \Cref{th-rec/trans}  is equivalent to the following propositions that we will prove independently in the next sections. 

\smallskip
\begin{prop.} \label{th-rec}
If the geodesic flow on $X$ is recurrent and ergodic, then it is also the case of the $\mu$-walk on $X$.

\end{prop.}

\smallskip

Denote $\mathscr{P}^C(X)$  the collection of subsets $F\subseteq X$ such that $F$ is relatively compact and has positive Haar measure. 

\begin{prop.} \label{th-trans} 
 If  the Green functions of the geodesic flow $G(.,F)_{F\in \mathscr{P}^C(X) }$ are  locally integrable, then it is also the case of   the Green functions $G_{\mu}(.,F)_{F\in \mathscr{P}^C(X) }$ of the $\mu$-walk on $X$. 
 \end{prop.}

For the proofs to come, it will be useful to embed $G$ in a linear group $SL(V_{0})$ via a faithful irreducible proximal algebraic representation.   
$V_{0}$ will  be endowed with a $K$-invariant scalar product $\langle.,. \rangle_{0}$ such that $\mathfrak{a}$ is self-adjoint \cite[Lemma 6.33]{BQRW}. 

\subsection{Renewal theory}

It happens that both Propositions \ref{th-rec} and \ref{th-trans} rely on  renewal results for the Cartan projection of the right random walk on $G$.  The role of this section is to state and prove these results.

We first give some context. A renewal theorem considers a transient random walk and estimates the average time spent in a given bounded subset when the latter degenerates.  The standard case of a non-arithmetic walk on $\R$ can be found in \cite{lpsm-cours}. It was generalized by Kesten in \cite{Kes} to the Iwasawa cocycle for linear random walks.

\begin{th.*}[Renewal Theorem for the Iwasawa cocycle,  \cite{Kes}]
Let $d\geq2$ and $m$ be a probability measure on $SL_{d}(\R)$ with a finite first moment and such that $\Gamma_{m}:=\langle \text{supp\,} m \rangle$ is strongly irreducible and unbounded. Denote by $\lambda_{m}>0$  the first Lyapunov exponent of $m$ and by $(S_{n})_{n\geq 0}$ the left $\mu$-random walk on $SL_{d}(\R)$ starting at $\emph{Id}$. Then, for any interval $I\subseteq \R$ and vector $v\in \R^d\smallsetminus\{0\}$, 
$$\mathbb{E}(\sharp\{n\geq 0, \,\log ||S_{n}v|| \in I+t))\, \underset{t\to +\infty}{\longrightarrow}\,\frac{\emph{leb}(I)}{\lambda_{m}}$$

\end{th.*}




\noindent{\bf Remark}. A more precise statement is proven in \cite{Gui-Lep}, and the speed of convergence is estimated in \cite{Li} under the assumption that $m$ has an exponential moment.

\bigskip

Now consider the $\mu$-walk on our rank one simple Lie group $G$. The assumption that $\mu$ has a finite first moment implies that   for $\beta$-almost every $b\in B$, 
$$t_{n}(b)\underset{n\to +\infty}{\sim} n \lambda_{\mu} $$ 
where $\lambda_{\mu}>0$ is  the first \emph{ Lyapunov exponent} of $\mu$ (see \cite{BQRW}).  In view of the above renewal theorem, it is natural to conjecture the following renewal statement for the Cartan projection :  
\begin{align}\label{eq5}
\mathbb{E}_{\beta}(\sharp \{n \geq 0, \,t_{n}\in I +t\}) \, \underset{n\to +\infty}{\longrightarrow} \,\frac{\text{leb}(I)}{\lambda_{\mu}} 
\end{align}
This is known to be true if $\mu$ has a finite exponential moment \cite{Li} but the case where $\mu$ has only a finite first moment is still open. We prove two  propositions (\ref{ren-lem1},\ref{ren-lem2}) that can be seen as first steps to show  the convergence (\ref{eq5}).  

\bigskip

The first proposition guarantees that for any  point $x\in X$ which has a recurrent  orbit $(xa_{t})_{t>0}$ under the geodesic flow,  the sequence $(x a_{t_{n}(b)})_{n\geq 0}$ is also recurrent for $\beta$-almost every $b\in B$ (see \Cref{rec-tn}).

\begin{prop.}\label{ren-lem1}
Let $I\subseteq \R$ be a large enough bounded interval.  For any subset $S\subseteq \R_{+}$  containing arbitrary large real numbers, 
 for $\beta$-almost every $b\in B$, 
$$\sharp \{ n\geq 0, \,\,t_{n}(b)\in I+S\} =+\infty$$

\end{prop.}

\bigskip
The proof \Cref{ren-lem1} relies on \Cref{ren-prelem1}, according to which the probability that the Cartan projection of a $\mu$-trajectory $(gb_{1}\dots b_{n})$ meets the translate $I+s$ is close to 1 as long as $s$ is large enough. 
Recall from \ref{sec-not/strat} that $t_{g}\geq 0$ denotes the Cartan projection of an element $g\in G$. 

\begin{lemme} \label{ren-prelem1}
Let $\varepsilon>0$ and $I\subseteq \R$   a large enough bounded interval. Then for every $g\in G$,
$$\liminf_{s\to +\infty} \beta \{b\in B,\,  \exists n\geq 0,\, t_{gb_{1}\dots b_{n}} \in  I+s \}>1-\varepsilon $$

\end{lemme}

\smallskip
\begin{proof}[Proof of \Cref{ren-prelem1}]
We use the represention $G\subseteq SL(V_{0})$ introduced earlier in \Cref{rec-crit}. Our assumptions on the scalar product $\langle.,.\rangle_{0}$ of $V_{0}$  imply that the product $gb_{1}\dots b_{n}$ and its adjoint  $$S_{n}(b,g)= {^tb_{n}}\dots {^t b_{1}} {^tg}$$ have the same Cartan projection $t_{gb_{1}\dots b_{n}}$. Hence  the norm of  $S_{n}(b,g)$ seen as an operator on $V_{0}$ is of the form $||S_{n}(b,g)||=e^{c_{0}t_{gb_{1}\dots b_{n}}}$ where $c_{0}=\log||a_{1}||>0$.   \Cref{ren-prelem1} can then be restated as :  for every $\varepsilon>0$, there exists $I\subseteq \R$ bounded interval such that for all $g\in G$, 
\begin{align} \label{eq6}
\liminf_{s\to +\infty} \beta\{b\in B, \exists n\geq 0,\, \log||S_{n}(b,g)|| \in  I+s \}>1-\varepsilon 
\end{align}

We know by \cite[(1.17)]{Kes}  this statement is true for the Iwasawa cocycle :  there exists a bounded interval $J$ such that for all  $g\in G$, $v\in V\smallsetminus\{0\}$, for all $s>s_{g,v}$,   
 \begin{align}\label{eq7}
\beta \{b\in B, \exists n\geq 0,\, \log||S_{n}(b,g) v|| \in  J+s \}>1-\varepsilon/2
\end{align}

Moreover, arguing by contradiction, we can infer from \cite[Corollary 4.8]{BQRW} that the difference between the Iwasawa cocycle and the Cartan projection  is ultimately bounded : there exists constants $R>0$ such that for all unit vector $v\in V_{0}$
 \begin{align*}
\beta\{b\in B,  \forall n\geq 0, \, 
 \log||{^tb_{n}}\dots {^t b_{1}}||- \log||{^tb_{n}}\dots {^t b_{1}} v||  < R \} > 1-\varepsilon/2 
\end{align*}
In particular, choosing for each $g\in G$  a  unit vector $v_{g}\in V_{0}$ such that  $||{^tg}v_{g}||\geq \frac{1}{2}||{^tg}||$, and setting $R' = R+\log 2$,
\begin{align} \label{eq8}
\beta\{b\in B, \forall n\geq 0, \, 
 \log||S_{n}(b,g) ||- \log||S_{n}(b,g) v_{g}||  < R' \} > 1-\varepsilon/2
\end{align}

Consider now an interval $I$ that contains the $R'$-neighborhood of $J$. Then using (\ref{eq7}) and (\ref{eq8}), for all $g\in G$, $s>s_{g,v_{g}}$,{
\begin{align*}
& \beta\{b\in B, \exists n\geq 0,\, \log||S_{n}(b,g)|| \in  I+s \} \\
&\geq  \beta \big\{b\in B, \exists n\geq 0,\, \log||S_{n}(b,g) v_{g}|| \in  J+s  \\
 & \,\,\,\,\,\,\,\,\,\,\,\,\,\,\,\,\,\,\,\,\,\,\,\,\,\,\,\,\,\,\,\,\,\,\,\,\,\,\text{ and }   \log||S_{n}(b,g) ||- \log||S_{n}(b,g) v_{g}||  < R' \big\} \\
&\geq 1-\varepsilon
\end{align*}}
 We have finally obtained (\ref{eq6}), hence the lemma.

\end{proof}

\bigskip

\begin{proof}[Proof of \Cref{ren-lem1}]
Let $\varepsilon\in ]0,1[$ and $I\subseteq \R$ as in \Cref{ren-prelem1}. Let $S\subseteq \R_{+}$ be a subset containing arbitrarily large real numbers.  Define by induction a family of stopping times $(n_{k}: B\rightarrow \N)_{k\in \N}$ as follows :

\begin{itemize}
\item $n_{0}=0$ 
\item
Applying \Cref{ren-prelem1} with $g=b_{1}\dots b_{n_{k}(b)}$,  choose $n_{k+1}(b)>n_{k}(b)$ for which
$$\beta\{a\in B,  \exists n\in \rrbracket  n_{k}(b), n_{k+1}(b) \rrbracket  , \, t_{b_{1}\dots b_{n_{k}(b)} a_{n_{k}(b)+1} \dots a_{n} }\in I+S\}>1-\varepsilon $$ 
and such that  $n_{k+1}$ is a measurable function of the product $b_{1}\dots b_{n_{k}(b)}$. 
\end{itemize}

Now observe that if $1\leq k_{1}<\dots< k_{N}$ are distinct integers, then by the Markov property 
$$\beta\{b\in B,  \forall  i \in \rrbracket 1, N \rrbracket,  \forall n \in \rrbracket n_{k_{i}-1},   n_{k_{i}}\llbracket , \, t_{n}(b)\notin I+S\} \leq (1-\varepsilon)^N $$
Hence, given any infinite sequence of integers $1\leq k_{1}<\dots <k_{i}<\dots$, 
$$\beta \{b\in B,  \forall  i \geq 1, \forall n\in \rrbracket n_{k_{i}-1},   n_{k_{i}}\rrbracket ,  \, t_{n}(b)\notin I+S\}=0 $$
This equality implies the statement of the lemma.
\end{proof}

\bigskip
The second proposition will be used to compare the Green functions of the geodesic flow and of  the $\mu$-walk on $X$.  It states that the average time spent  by the Cartan projection $(t_{n}(b))$ of a trajectory $(b_{1}\dots b_{n})$  in a given interval $I$ of $\R$ is bounded by a constant that only depends on  $\leb(I)$.
\begin{prop.} \label{ren-lem2}

Let $I\subseteq \R$ be a bounded interval. Then
$$\sup_{t\in \R}\,\mathbb{E}_{\beta}(\sharp \{n \geq 0, \,t_{n}\in I +t\})  <\infty$$ 

\end{prop.}

\smallskip
The idea is that the Cartan projection of a trajectory $(b_{1}\dots b_{n})_{n\geq 0}$ has a low probability to come back to an interval once it has gone past it. We formalize this in \Cref{lemme-devi}. 
\begin{lemme}\label{lemme-devi}
 \emph{Let $R,\varepsilon>0$. There exists $n_{0}\geq0$ such that for every $g\in G$, 
$$\beta \{b\in B,\,\forall n \geq n_{0},\, t_{gb_{1}\dots b_{n}}\geq t_{g}+R\}\geq 1-\varepsilon$$}
\end{lemme}

\smallskip
\begin{proof}[Proof of \Cref{lemme-devi} ]
We use again the represention $G\subseteq SL(V_{0})$ introduced earlier in the section. Our assumptions on the scalar product $\langle.,.\rangle_{0}$ of $V_{0}$  imply that the product $gb_{1}\dots b_{n}$ and its adjoint  $S_{n}(b,g)={^tb_{n}}\dots {^t b_{1}} {^tg}$ have the same Cartan projection $t_{gb_{1}\dots b_{n}}$. Hence  
$$||S_{n}(b,g)||=e^{c_{0}t_{gb_{1}\dots b_{n}}}$$ 
where $c_{0}=\log ||a_{1}||>0$.  \Cref{lemme-devi} can then be rephrased as :
\emph{for any $C, \varepsilon>0$, there exists $n_{0}\geq0$ such that for all $g\in G$, 
 \begin{align} \label{eq10}
 \beta\{b\in B, \forall n \geq n_{0},\,  || S_{n}(b,g)||\geq C||{^tg}||\}\geq 1-\varepsilon 
 \end{align}}
To prove (\ref{eq10}), assume by contradiction there exist  $C, \varepsilon >0$, a sequence of integers $(N_{k})\to \infty$ and elements $(g_{k})\in G$  such that for all $k\geq0$, 
\begin{align} \label{eq11}
\beta\{b\in B, \exists n \geq N_{k},\,   || S_{n}(b,g_{k})||< C||{^tg_{k}}|| \}\geq \varepsilon 
\end{align}
Up to extraction, one may also suppose that the normalized sequence $(\frac{^tg_{k}}{||{^tg_{k}}||})$ converges to an endomorphism $f_{\infty}\in \text{End}(V_{0})$.
By (\ref{eq11}), there exists a set $B'\subseteq B$ of measure at least $\varepsilon$ such that for every $b\in B'$, there are sequences of integers $(k_{i})$, $(n_{i})$ going to infinity and satisfying 
$$ || S_{n_{i}}(b,g_{k_{i}})||< C||{^tg_{k_{i}}}||$$
leading to 
$$ || {^tb}_{n_{i}}\dots {^tb_{1}} f_{\infty}|| <  || {^tb}_{n_{i}}\dots {^tb_{1}} (f_{\infty}-\frac{{^tg_{k_{i}}}}{||{^tg_{k_{i}}}||} )|| +C = o(||{^tb}_{n_{i}}\dots {^tb_{1}}||)$$
where the last equality is true for almost every $b$. But this yields a contradiction with  \cite[Corollary 4.8]{BQRW}. Hence we have (\ref{eq10}), and the lemma follows.

\end{proof}

We can now conclude the section with the proof of \Cref{ren-lem2} 
\begin{proof}[Proof of \Cref{ren-lem2} ]
Denote by $N_{I}: B\rightarrow \N\cup\{\infty\}, b\mapsto \sharp\{n\geq 0, \,t_{n}(b)\in I\}$ the function that counts the time spent in $I$ for the Cartan projection of a $\mu$-trajectory on $G$.  We want to bound above the expectation of  $N_{I}$.  Let  $R>\leb(I)$,  $\varepsilon\in ]0,1[$, and $n_{0}\geq 0$ as in  \Cref{lemme-devi}. We are going to show that for all $k\geq 0$, 
\begin{align}\label{eq12}
\beta \{b\in B, N_{I}(b)\geq kn_{0}+1\}\leq \varepsilon^k 
\end{align}

Once (\ref{eq12}) is established, it is easy to conclude : 
$$ \int_{B}N_{I} \,d\beta=\sum_{n\geq1} \beta(N_{I}\geq n)\leq \sum_{k\geq 0} n_{0}\beta(N_{I}\geq kn_{0}+1) \leq n_{0} \frac{1}{1-\varepsilon} $$
and the constant $n_{0} \frac{1}{1-\varepsilon}$ depends on  $I$ only via the choice of $R$ which is solely bound to satisfy $R>\leb(I)$. 

\bigskip
Let us now prove (\ref{eq12}). We introduce a sequence of stopping times $(\tau_{i} : B\rightarrow \N\cup\{\infty\})_{i\geq 1}$ indicating the first hitting time of the interval $I$ by  the sequence  $(t_{n}(b))_{n\geq 0}$, then its successive return times separated by at least $n_{0}$ steps.   $$\tau_{1}=\inf\{n\geq 0,\, t_{n} \in I\}, \,\,\,\,\,\tau_{i+1}=\inf\{n\geq \tau_{i}+n_{0},\, t_{n}\in I \}$$

Observe that
\begin{align*}
\beta\{b\in B, N_{I}(b)\geq kn_{0}+1\}&\leq \beta\{b\in B, \, \tau_{k+1}(b) <\infty \}\\
&=\sum_{{\bf j}\in \N^k } \beta \{b\in B, \, (\tau_{i}(b))_{i\leq k}=\textbf{j}, \tau_{k+1}(b)<\infty\}
\end{align*}

 The inequality  $R>\leb(I)$ yields for every $\textbf{j}=(j_{1},\dots, j_{k})\in \N^k$ the inclusion  
$$\{b\in B,\, (\tau_{i}(b))_{i\leq k}=\textbf{j}, \tau_{k+1}(b)<\infty \}\subseteq \{b\in B,\, (\tau_{i}(b))_{i\leq k}=\textbf{j} \text{ et }\exists n\geq j_{k}+n_{0}, t_{n}(b) < t_{j_{k}}(b)+R \}$$

Using the  Markov property and  \Cref{lemme-devi}, we infer that 
$$\beta \{b\in B, \,(\tau_{i}(b))_{i\leq k}=\textbf{j}, \tau_{k+1}(b)<\infty    \}\leq \beta \{b\in B, \,(\tau_{i}(b))_{i\leq k}=\textbf{j}\}\, \varepsilon$$

Summing over every $\textbf{j}\in \N^k$, and iterating the process, we obtain 
\begin{align*}
  \beta \{b\in B, \, \tau_{k+1}(b)<\infty\}&\leq  \beta\{b\in B, \, \tau_{k}(b)<\infty\}\, \varepsilon \leq \dots \leq \varepsilon^k
\end{align*}

Hence, as announced,
 $$\beta\{b\in B, N_{I}(b)\geq kn_{0}+1\}\leq \varepsilon^{k}$$

 \end{proof}

\subsection{Recurrence and ergodicity}  \label{sec-recerg}

In this section we prove  \Cref{th-rec} : we assume  the geodesic flow $(a_{t})_{t\in \R}$ on $X$ to be recurrent ergodic and show  that  the $\mu$-random walk on $X$ is recurrent ergodic as well.

\subsubsection{Recurrence}

We begin with the recurrence of the walk. 
\bigskip

\begin{lemme}\label{rec}
 The $\mu$-walk on $X$ is recurrent
 \end{lemme}
 
 \bigskip
 Consider a large compact set $\widetilde{L}\subseteq X$. We aim to show that for almost every $x\in X$, and $\beta$-almost every $b\in B$, there exist infinitely many times $n\geq0$ such that 
 $$xb_{1}\dots b_{n} \in \widetilde{L}$$
 
Endow $G$ with a left invariant Riemannian metric, $X$ with the quotient metric, denote by  $\xi^-_{0}=P^-/P^-$  the base point of the flag variety $\PP^-$. According to  \Cref{distform}, there exists  a constant $C_{0}>0$ such that for $\beta$-almost every $b\in B$,  large enough $n\geq 0$, 
 $$d(b_{1}\dots b_{n}, k_{\infty}(b)a_{t_{n}(b)})  \leq C_{0}d(b_{1}\dots b_{n}, F(\xi_{0}^-, \xi_{b}))+C_{0}$$

Fix a compact subset $L\subseteq X$ and a constant $R>0$ (to be specified below), and assume  $\widetilde{L}$ contains the $C_{0}(R+1)$-neighborhood of $L$.  In this case, we just need to show that for almost every $x\in X$, $\beta$-almost every $b\in B$, there exists infinitely many times $n\geq 0$ such that 
\begin{align}\label{eq13}
xk_{\infty}(b)a_{t_{n}(b)} \in L \,\,\, \,\,\,\text{and}  \,\,\, \,\,\,d(b_{1}\dots b_{n}, F(\xi^-_{0}, \xi_{b}))\leq R
\end{align}

The difficulty is that the set of return times in $L$ given by 
$$\{n\geq 0, \,xk_{\infty}(b)a_{t_{n}(b)} \in L\}$$
 has null density in $\N$, hence we can not say directly that it intersects   $$\{ n\geq 0, \,d(b_{1}\dots b_{n}, F(\xi^-_{0}, \xi_{b}))\leq R\}$$ 
even if the latter has a density close to one (by the proof of \Cref{th-traj-geod}).  

\bigskip
A first important observation is  that we can ignore the term $k_{\infty}(b)$. More precisely, using Fubini's Theorem and equation \ref{eq4}, the statement (\ref{eq13}), hence \Cref{rec}, reduces to the following.

\smallskip
\begin{lemme}\label{rec-reduction}
We can choose the parameters $(L, R)$ such that for almost every $x\in X$, $\beta$-almost every $b\in B$,  infinitely many times $n\geq 0$,
$$xa_{t_{n}(b)}  \in L \,\,\, \,\,\,\text{and}  \,\,\, \,\,\,d(e, F(b^{-1}_{n}\dots b^{-1}_{1}\xi^-_{0}, \xi_{T^nb}))\leq R $$
\end{lemme}

\smallskip
This reduction is \emph{crucial}  because it separates the effects of the $n$ first instructions $(b_{1},\dots, b_{n})$ and of the tail $T^nb$, thus allowing to argue conditionally to the situation at time $n$. To show \Cref{rec-reduction}, we use the following strategy : Prove that for almost every $(x,b)$, the sequence $(xa_{t_{n}(b)})_{n\geq 0}$ meets $L$ infinitely often (\Cref{rec-tn}). Show that among those $n$, infinitely many  satisfy 
$$d(e, F(b^{-1}_{n}\dots b^{-1}_{1}\xi^-_{0}, \xi_{T^nb}))\leq R$$
To obtain the latter, we justify in \Cref{remp} that we may replace  $\xi_{T^n b}$ by the term $b_{n+1}\dots b_{n+k_{n}}\xi_{n}$ where $k_{n}\geq0$ is a large integer, and $\xi_{n}$ is a random point on the flag variety $\PP$, then  we use the Markov property  to conclude (together with \Cref{k0}).

\bigskip

Let us begin with the statements and  proofs of the three lemmas advertised at the moment. 
\begin{lemme}\label{rec-tn}
We can choose the compact set $L\subseteq X$ such that for almost every $x\in X$, $\beta$-almost every $b\in B$, there exists infinitely many times $n\geq 0$ for which
$$xa_{t_{n}(b)}  \in L$$
\end{lemme}

\begin{proof}
Let us first specify the compact set $L$. According to \Cref{ren-lem1}, there exists a constant $c>0$ such that if $I=[0,c]$ and $(s_{k})\to +\infty$ then for $\beta$-almost every $b\in B$, 
\begin{align}\label{eq14}
\sharp  \{n \geq 0, \,\,t_{n}(b) \in \bigcup_{k\geq 0} I+s_{k}\} =+\infty 
\end{align}
Set $c'=\max_{|t|\leq c} d(e,a_{t})$ where $d$ refers to the metric on $G$.   Fix some point $y_{0}\in X$ and set $L=\{y\in X, \, d(y_{0}, y) \leq c'+1\}$ the set of vectors $y\in X$ whose distance (in $X$) to $y_{0}$ is less than $c'+1$. 

 Let $E\subseteq X$ be the set of elements $x\in X$ such that there exists a sequence of real numbers $(s_{k})\rightarrow +\infty$ for which $(xa_{s_{k}})\rightarrow y_{0}$. The assumption that the geodesic flow on $X$ is recurrent ergodic implies that $E$ has full measure in $X$.  Let $x\in E$. Then the set 
 $$\{t >0, \,xa_{t}\in L\}$$
 contains a subset of the form $\bigcup_{k\geq 0} I+s'_{k}$ where $(s'_{k})\rightarrow+\infty$. Hence, by (\ref{eq14}), for $\beta$-almost every $b\in B$, the sequence $ (xa_{t_{n}(b)} )$ meets $L$ infinitely often.

\end{proof}

\bigskip

\begin{lemme} \label{remp}
Let $\delta>0$.  There exists a sequence of integers  $(k_{n})\in \N^{\N^\star}$ such that for $\beta$-almost-every $b\in B$ and $\nu_{\PP}^{\otimes \N^\star}$-almost every $(\xi_{n})\in\PP^{\N^\star}$, for large enough $n\geq 0$, 
$$d(b_{n+1}\dots b_{n+k_{n}}\xi_{n}, \xi_{T^nb}) \leq \delta$$
\end{lemme}

\begin{proof}

We have for $\beta$-almost every $b\in B$,  
$$(b_{1}\dots b_{n})_{\star} \nu_{\PP} \underset{n\to +\infty}{\longrightarrow} \nu_{\PP, b} =\delta_{\xi_{b}}$$
hence
$$\nu_{\PP}\{\xi \in \PP,  \,d(b_{1}\dots b_{n}\xi, \xi_{b}) > \delta\} \underset{n\to +\infty}{\longrightarrow} 0$$
Integrating in $b\in B$, we obtain 
$$\beta \otimes \nu_{\PP}\{(b,\xi) \in \PP,  \,d(b_{1}\dots b_{n}\xi, \xi_{b}) > \delta\} \underset{n\to +\infty}{\longrightarrow} 0$$
Extracting a subsequence whose sum is finite, we obtain  $(k_{n})\in \N^{\N^\star}$ such that
\begin{align*}
\sum_{n\geq 0} \beta \otimes \nu_{\PP}^{\otimes \N^\star} \{(b, (\xi_{i})) \in B\times \PP^{\N^\star}, \,\, d(b_{1}\dots b_{k_{n}}\xi_{n}, \xi_{b}) > \delta\} <\infty
\end{align*}
The observation that  $d(b_{1}\dots b_{k_{n}}\xi_{n}, \xi_{b})$ and $d(b_{n+1}\dots b_{n+k_{n}}\xi_{n}, \xi_{T^nb})$ have the same law, combined with  Borel-Cantelli Lemma, lead to the statement in  \Cref{remp}. 
\end{proof}

\begin{lemme} \label{k0}
There exists a constant $R'>0$ such that  for every  $\xi^-\in \PP^-$, 
$$\nu \{\xi \in \PP,\, d(e, F(\xi^-,  \xi))\leq R'\}> 2/3$$
\end{lemme}

\begin{proof}
We know from \Cref{lem-xib} that for every $\xi^-\in \PP^-$, there exists a constant $R_{\xi^-}>0$ such that  
\begin{align}\label{eq15}
\nu \{\xi \in \PP,\, d(e, F(\xi^-,  \xi))\leq R_{\xi^-}\}> 2/3
\end{align}
We need to show that $R_{\xi^-}$ may be chosen independently of $\xi^-$.  To see this, notice that the function 
$$\FF^+ \rightarrow [0,+\infty[, (\xi^-,\xi)\mapsto d(e, F(\xi^-, \xi))$$
is continuous and proper, as it can be identified with the quotient  map $G/Z_{G}(\mathfrak{a})\rightarrow  [0,+\infty[, gZ_{G}(\mathfrak{a})\mapsto  d(g^{-1}, \exp(\mathfrak{a})K)$. In particular for all $C>0$, Heine Theorem gives a constant $\delta>0$ such that  for all $(\xi^-,\xi), (\eta^-, \eta)\in \FF^+$ with $d(\xi^-,\eta^-)\leq \delta$, $ d(\xi,\eta)\leq \delta$, 
\begin{align}\label{eq16}
 d(e, F(\xi^-, \xi))\leq C\implies  d(e, F(\eta^-, \eta))\leq C+1 
\end{align}
(\ref{eq15})  and (\ref{eq16}) together imply that the constant $R_{\xi^-}$ can be chosen uniformly on a neighborhood of $\xi^-$. The compactness of $\PP^-$ then leads to a uniform constant $R'>0$ as in the statement of the lemma.

\end{proof}

We now prove \Cref{rec-reduction}.

\smallskip
\begin{proof}[Proof of \Cref{rec-reduction}]

We first make preparations to replace later the term $\xi_{T^nb}$ by $b_{n+1}\dots b_{n+k_{n}}\xi_{n}$ where $k_{n}\geq0$ is a large integer and $\xi_{n}$ is a random point on the flag variety $\PP$. 
Let $R'>0$ as in \Cref{k0}, set $R=R'+1$. As we saw in the proof of \Cref{k0}, there exists a constant $\delta>0$ such that for all $(\xi^-, \xi')\in \FF^+$, all $\xi \in \PP$ with $d(\xi', \xi)\leq \delta$, one has 
$$ d(e, F(\xi^-, \xi'))\leq R'\implies d(e, F(\xi^-, \xi))\leq R$$
Choose a sequence $(k_{n})\in \N^{\N^\star}$ as in \Cref{remp}. 
\bigskip

We now proceed to the proof.  Let $L\subseteq X$ be as in \Cref{rec-tn} and fix a vector $x\in X$ such that for almost every $b\in B$, the set $\mathscr{N}_{x,b}:=\{n\geq 0, \,xa_{t_{n}(b)}\in L\}$ has infinite cardinal.  Define  by induction a sequence of stopping times $\tau_{i} : B\rightarrow \N\cup \{\infty\}$ setting
$$\left\{
    \begin{array}{ll}
     \tau_{1}(b):= \inf \{n \geq 0, \,n\in \mathscr{N}_{x,b}\} \\
     \\
      \tau_{i+1}(b):=\inf\{n\geq \tau_{i}(b) +k_{\tau_{i} (b)}+1, \,  n\in \mathscr{N}_{x,b}\} 
    \end{array}
\right.$$

Given some integers $i_{1}>i_{0}\geq 0$, one has by the Markov property and \Cref{k0}
\begin{align*}
(\frac{1}{3})^{i_{1}-i_{0}+1}\geq \beta\otimes \nu^{\otimes \N^\star}_{\PP} \{&(b, (\xi_{i})) \in B\times \PP, \,\, \forall i \in \llbracket i_{0},i_{1} \rrbracket,\,\, \\
&d(e, \,F(b^{-1}_{\tau_{i}(b)}\dots b^{-1}_{1}\xi^-_{0},    b_{\tau_{i}(b)+1}\dots b_{\tau_{i}(b)+k_{\tau_{i}(b)}}\xi_{i} ))>R'\} 
\end{align*}
Letting $i_{1}$ go to $+\infty$, we deduce that for $\beta$-almost every $b\in B$, there exists $i\geq i_{0}$ such that 
$$ d(e, \,F(b^{-1}_{\tau_{i}(b)}\dots b^{-1}_{1}\xi^-_{0},    b_{\tau_{i}(b)+1}\dots b_{\tau_{i}(b)+k_{\tau_{i}(b)}}\xi_{i})) \leq R' $$

As $i_{0}$ can be chosen arbitrarily large, we obtain that for almost-every $b\in B$, almost every $(\xi_{i})\in \PP^{\N^\star}$,  there exists infinitely many  integers $i\geq 0$ such that 
\begin{align}\label{eq17}
d(e, \,F(b^{-1}_{\tau_{i}(b)}\dots b^{-1}_{1}\xi^-_{0},    b_{\tau_{i}(b)+1}\dots b_{\tau_{i}(b)+k_{\tau_{i}(b)}}\xi_{i} ))\leq R' 
\end{align}

But our choice of $(k_{n})$ guarantees that for large enough $i\geq0$, 
\begin{align}\label{eq18}
d(b_{\tau_{i}(b)+1}\dots b_{\tau_{i}(b)+k_{\tau_{i}(b)}}\xi_{i},  \xi_{T^{\tau_{i}(b)}} )\leq \delta 
\end{align}

By (\ref{eq17}),  (\ref{eq18})  and the definition of $\delta$, we can conclude : for   almost-every $b\in B$, there exists infinitely many  integers $i\geq 0$ such that 
$$ d(  e, \, F(b^{-1}_{\tau_{i}(b)}\dots b^{-1}_{1}\xi^-_{0},  \xi_{T^{\tau_{i}(b)}}) \leq R $$

This finishes the proof of \Cref{rec-reduction}, yielding  \Cref{rec}.

\end{proof}

\bigskip
\subsubsection{Ergodicity}

We now prove the ergodicity of the $\mu$-walk on $X$. 

\bigskip

\begin{lemme}\label{erg}
The $\mu$-walk on $X$ is ergodic. 
\end{lemme}

\bigskip

The key idea is that the subgroup $\Gamma_{\mu}$ generated by the support of $\mu$ must contain loxodromic elements, whose action on $X$ is (almost) conjugate to the geodesic flow, hence ergodic.  Recall that an element $g_{0}\in G$ is \emph{loxodromic} if it can be written, up to conjugation,  as  $g_{0}=ma_{c}$ where $m\in K$, $c>0$, and $ma_{c}=a_{c}m$  (see also \cite[Section 6.10]{BQRW}).

\begin{lemme}\label{lox}
The action of a loxodromic element $g_{0}$ on $X$ is conservative ergodic for the Haar measure. 
\end{lemme}

In this statement the conservativity of $g_{0}$ means that for almost-every point $x\in X$, the sequence $(xg^n_{0})_{n\geq 0}$ subconverges to $x$ (see \cite[Section 1.1]{Aar} for more details). 

\begin{proof}[Proof of \Cref{lox}]
One can assume that $g_{0}=ma_{c}$ where $m\in K$, $c>0$, and $ma_{c}=a_{c}m$.  In particular, the recurrence of the geodesic flow on $X$ implies the conservativity of $g_{0}$. Its  ergodicity  follows by standard arguments (given for the geodesic flow  in \cite[Theorem 7.4.3]{Aar}). We  explain them briefly. Denote by $\lambda$ a Haar measure on $X$, let  $f, p\in L^1(X,\lambda)$ with $p>0$, $\lambda(p)=1$. Hopf Ergodic Theorem \cite[2.2.5]{Aar} and the conservativity of $g_{0}$ imply the almost-sure convergence  :
$$\frac{\sum_{k=0}^{n-1}f(. g^k_{0})}{\sum_{k=0}^{n-1}p(. g^k_{0}) }\,\underset{n\to \pm \infty}{\longrightarrow}\, \underbrace{\mathbb{E}_{p\lambda}(\frac{f}{p}|\mathcal{I})}_{\Phi_{f,p}}$$
where $ \mathbb{E}_{p\lambda}(\frac{f}{p}|\mathcal{I})$ is the conditional expectation of 
$f/p$ for the probability measure $p\lambda$  and with respect to the $\sigma$-algebra  $\mathcal{I}$ of the  $\lambda$-a.e. $g_{0}$-invariant subsets of $X$. We need to show this $\sigma$-algebra is $\lambda$-trivial, which amounts to say that for every choice of $f,p$, the limit $\Phi_{f,p}$ is $\lambda$-a.e. constant. Endow  $G$  with a Riemannian metric that  is $G$-left invariant and $K$-right invariant, and equip $X$ with the quotient metric. Arguing as in \cite[7.4.3]{Aar} we can assume that $p$, then $f$, are regular enough so that $\Phi_{f,p}$ is constant along the stable or unstable manifolds of $g_{0}$. More precisely, denote by $U\subseteq G$ (resp. $U^{-}$)  unipotent connected subgroup of   $G$ whose Lie algebra is $\mathfrak{u}$ (resp. $\mathfrak{u}^-$). Then for $x\in X$, $u\in U$,  
$$d(xug_{0}^k, xg_{0}^k) =d(xua^k_{c}, xa^k_{c})\underset{k\to + \infty}{\longrightarrow}0$$
and the same goes for $U^-$ and $k\to-\infty$. By our choice of  $p$ and $f$, this yields for every $u \in U\cup U^-$ the almost-sure equality 
\begin{align*}
\Phi_{f,p}(. u)=\Phi_{f,p} \tag{$\lambda$-a.e.}
\end{align*}
As $U$ and $U^-$ together generate $G$, the $\lambda$-a.e. invariance  of $\Phi_{f,p}$ by $U$ and $U^-$ implies its $\lambda$-a.e. invariance by a countable dense subset of  $G$, hence by $G$. The map  $\Phi_{f,p}$ is then  $\lambda$-a.e. constant. 
\end{proof}

\bigskip
\begin{proof}[Proof of \Cref{erg}]
The subroup $\Gamma_{\mu}$ generated by  the support of $\mu$ is Zariski-dense in $G$, so it must contain some loxodromic element $g_{0}$ (see \cite[Prop. 6.11]{BQRW}). By \Cref{lox}, the action of $g_{0}$ on $X$ is ergodic, hence so is  the action of $\Gamma_{\mu}$. This proves the ergodicity of the $\mu$-walk on $X$. 

\end{proof}

\newpage
\subsection{Transience}

In this section we prove \Cref{th-trans} :

\begin{prop-trans} 
 If the Green functions of the geodesic flow $G(.,F)_{F\in \mathscr{P}^C(X) }$ are all locally integrable, then it is also the case of   the Green functions $G_{\mu}(.,F)_{F\in \mathscr{P}^C(X) }$ of the $\mu$-walk on $X$. 
 \end{prop-trans}

\begin{proof}

Let $E,F\subseteq X$ be compact $K$-invariant subsets of $X$ and $\lambda$ a Haar measure on $X$. We can write 

\begin{align}\label{eq19}
\int_{E}G_{\mu}(x,F)   \,d\lambda(x)&=  \int_{E}\int_{B}\sum_{n\geq 0}1_{F}(xb_{1}\dots b_{n}) \,d\beta(b)d\lambda(x) \nonumber \\
&= \int_{E}\int_{B}\sum_{n\geq 0}1_{F}(x a_{t_{n}(b)}) \,d\beta(b)d\lambda(x) 
\end{align}
where the last inequality comes from the $K$-invariance of $E$, $F$ and $\lambda$. 

Let $F'\subseteq X$ be a compact set such that $\bigcup_{t\in [0,1]}Fa_{t} \subseteq F'$. Then 
\begin{align}\label{eq20}
1_{F}(xa_{t_{n}(b)})\leq \int_{\R_{+}}1_{F'}(xa_{t})1_{[t_{n}(b), t_{n}(b)+1]}(t) dt 
\end{align}

Combining (\ref{eq19}) and (\ref{eq20}), we obtain 
\begin{align*}
\int_{E}G_{\mu}(x,F)   \,d\lambda(x)&\leq  \int_{E}\int_{B}\int_{\R_{+}}\sum_{n\geq 0}1_{F'}(x a_{t})1_{[t_{n}(b), t_{n}(b)+1]}(t) \,dt d\beta(b)d\lambda(x)\\
&= \int_{E}\int_{\R_{+}}1_{F'}(x a_{t}) [\int_{B}\sum_{n\geq 0}1_{[t_{n}(b), t_{n}(b)+1]}(t) d\beta(b)]  \,dtd\lambda(x)
\end{align*} 

 The term between brackets estimates the average time  spent  by the Cartan projection  of a $\mu$-trajectory on $G$ in the interval $[t-1, t]$. By \Cref{ren-lem2} it is less than a constant  $R\in ]0,+\infty[$ that  does not depend on  $t$  but only on the initial data $(G,K, \mathfrak{a}^+,  \mu)$. Finally, we get

\begin{align*}
\int_{E}G_{\mu}(x,F)   \,d\lambda(x)&\leq R \int_{E}\int_{\R_{+}}1_{F'}(xa_{t})   \,dtd\lambda(x)\\
&= R \int_{E}G(x,F')   \,d\lambda(x)\\
&<+\infty
\end{align*}

\end{proof}

\bigskip
 
\section{Appendix : Hopf dichotomy for the geodesic flow}
\label{Hopfdich}
In this appendix, we justify the following fact used in  \Cref{rec-crit}. The notations are those of \Cref{rec-crit}. In particular $G$ is a connected simple real Lie group of rank one, $\Lambda\subseteq G$ is a discrete subgroup, and $X=\Lambda\backslash G$.

\bigskip

\noindent{\bf Fact 2}
\emph{The geodesic flow $(a_t)_{t\in \R}$ on $X$ is  either recurrent  ergodic, or  transient with  locally integrable Green functions.}

\bigskip
This result is already known but usually stated differently using the notion of Poincaré series, as in \Cref{Kaimth} below. We explain here why Fact 2 is a reformulation of \Cref{Kaimth}. The point is that the Poincaré series of $\Lambda$ at the maximal exponent expresses, up to a multiplicative constant, the integral of the Green function of the geodesic flow on a $K$-orbit in $X$ (\Cref{poinc2}).  

\bigskip

Recall first the notion of Poincaré series. Endow $G/K$ with its standard structure of symmetric space, i.e. with its unique left $G$-invariant Riemannian metric that coincides with the Killing form on $T_{K/K}G/K\simeq \mathfrak{k}^\perp$. Write $d$ the corresponding distance map on $G/K$.  Given points $z_{1},z_{2}\in G/K$, and a positive real number $s>0$, the Poincaré series of  $\Lambda$ at $(z_{1},z_{2},s)$ is defined as 
$$\mathfrak{p}(z_{1},z_{2},s)=\sum_{g\in \Lambda}e^{-sd(z_{1},gz_{2})}$$

Observe that the convergence or divergence of the series $\mathfrak{p}(z_{1},z_{2},s)$ does not depend on the points $z_{1},z_{2}$ but only on the parameter $s$.  It is then natural to introduce the number 
$$\delta_{\Lambda}=\inf\{s>0, \,\mathfrak{p}(z_{1},z_{2},s)<\infty\} $$
known as  the critical exponent of $\Lambda$.  As $\Lambda$ is discrete,  $\delta_{\Lambda}$ is less or equal to  the exponential growth rate of the volume of balls in  $G/K$, given by
$$\delta_{G}=\lim_{R\to +\infty} \frac{1}{R}\log(V_{R})$$
where $V_{R}>0$ is the Riemannian volume of a ball of radius $R$ in $G/K$.
The case of equality $\delta_{\Lambda}=\delta_{G}$ expresses that the orbits of $\Lambda$ are not too sparse in $G/K$. As we see below, it is  a necessary condition for the geodesic flow on $X$ to be recurrent, but it is not sufficient in general. For instance, if $G=PSL_{2}(\R)$ and $X/K$ is a $\Z^d$-cover of a compact hyperbolic surface with $d\geq3$, then the geodesic flow on $X$ is transient \cite{Rees81} but $\delta_{\Lambda}=\delta_{PSL_{2}(\R)}=1$ \cite{CDST18}.  The following result claims that we can strengthen slightly the condition that $\delta_{\Lambda}$ is maximal to characterize the situations of recurrence/transience. It is usually called Hopf-Tsuji-Sullivan  Theorem, but it is actually due to Kaimanovich in the context of rank-one symmetric spaces. 

\begin{th.}\emph{\cite[Theorem 3.3]{Kai}} \label{Kaimth}
The geodesic flow $(a_{t})$ on $X$ is  recurrent ergodic  if and only if $\mathfrak{p}(z_{1},z_{2},\delta_{G})=+\infty$, and is  transient otherwise.
\end{th.}
\bigskip

\noindent{\bf Remark.} If  $G=Sp(1,m)$ for some $m\geq2$, or $G=F^{-20}_{4}$, then  \cite[Theorem 4.4]{Cor} claims that a discrete subgroup $\Lambda\subseteq G$ of infinite covolume satisfies $\delta_{\Lambda}<\delta_{G}$. In particular, by \Cref{Kaimth},  the geodesic flow on $X=\Lambda \backslash G$ is transient. According to  \Cref{th-rec/trans}, the same holds true for walks on $X$ given by a probability measure $\mu$ on $G$ with finite first moment and $\Gamma_{\mu}$ Zariski-dense  in $G$.   

\bigskip
We now explain why Fact 2 is a reformulation of \Cref{Kaimth}. We freely identify any subset of $G/K$ to a right $K$-invariant subset of $G$. Given $z\in G/K$, $\varepsilon>0$, we  denote by $B(z,\varepsilon)$ the ball of center $z$ and radius $\varepsilon$ in the symmetric space $G/K$. We also set $\lambda_{K}$ the Haar probability measure on $K$. Finally, given $s,t \geq 0$, $c>1$ we write $s=c^{\pm 1}t$ if $s\in [c^{-1} t,ct]$.  

\begin{lemme}\label{poinc1}
For all $\varepsilon >0$, there exists $c>1$ such that for $z_{1},z_{2}\in G/K$, $g_{1}\in z_{1}$, 
\begin{align}\label{eq21}
\int_{K}G(g_{1}k, B(z_{2},\varepsilon)) d\lambda_{K}(k) = c^{\pm 1} e^{-\delta_{G}d(z_{1}, z_{2})} 
\end{align}
\end{lemme}

\begin{proof}
Let $m$ be the Liouville measure on the symmetric space $G/K$. According to Helgason's book \cite[Theorem 5.8]{Hel}, there exists a constant $r>0$ such that for all non-negative function measurable functions $f : G/K\rightarrow \R_{+}$, 
\begin{align}\label{eq+}
\int_{G/K}f \,dm =r\int_{K\times \R^+} f(ka_{t} K) \sigma(t)\,d\lambda_{K}(k)dt 
\end{align}
where $\sigma(t)=\Pi_{\alpha\in \Phi^+} \sinh(\alpha (t v_{0}))^{\dim \mathfrak{g}_{\alpha}}$, with $v_{0}\in \mathfrak{a}^+$   unique vector of norm 1 (see beginning of \ref{rec-crit}). We must have the equivalence $\sigma(t)\sim r'e^{\delta_{G}t}$  for some $r'\in \{1/2, 1/4\}$ as $t$ goes to $+\infty$, and in particular, there exists a constant $R>0$ such that for all $t>1$, 
$$\sigma(t)= R^{\pm 1} e^{\delta_{G}t} $$ 

Let us now check (\ref{eq21}). We can assume that $g_{1}=e$ and $d(z_{1}, z_{2})> 1+\varepsilon$. 
Specifying $f$ in (\ref{eq+}) to be the characteristic function of the ball $B(z_{2}, \varepsilon)$ in $G/K$, we obtain

\begin{align*}
V_{\varepsilon}&= r\int_{K\times \R^+} 1_{B(z_{2}, \varepsilon)}(ka_{t} K) \sigma(t)\,d\lambda_{K}(k)dt\\
&= rR^{\pm 1}\int_{K\times \R^+} 1_{B(z_{2}, \varepsilon)}(ka_{t} K) e^{\delta_{G}t}\,d\lambda_{K}(k)dt\\
&= r(Re^\varepsilon)^{\pm 1} e^{\delta_{G}d(z_{2},0)}\int_{K}G(k, B(z_{2},\varepsilon)) d\lambda_{K}(k) 
\end{align*}
and finally, 
$$\int_{K}G(k, B(z_{2},\varepsilon)) d\lambda_{K}(k) = V_{\varepsilon} r^{-1}(Re^\varepsilon)^{\pm 1} e^{-\delta_{G}d(z_{2},0)}$$
\end{proof}

We now use the previous lemma to show that the Poincaré series of $\Lambda$  with parameter $s=\delta_{G}$ expresses the average of the Green function of the geodesic flow on a $K$-orbit in $X$.  Given a point $p\in X/K$ and $r>0$, denote by $B(p, r)$ the ball of center $p$ and radius $r$ in $X/K$ for the quotient metric, and define $r_{X}(p)$ as the supremum of the real numbers $r>0$ such that the preimage of $B(p, r)$ in $G/K$ is a collection of disjoint open balls (on wich $\Lambda$ acts transitively). In the case where $\Lambda$ has no torsion, the action of $\Lambda$ on such a collection of balls is simply transitive, and $r_{X}(p)$ is called the \emph{injectivity radius} of $X$ at $p$. In general, we only know that the action of $\Lambda$ has \emph{finite} stabilizer (by discreteness), and write $N_{X}(p)\in \N\smallsetminus\{0\}$ its cardinal.  

\begin{lemme} \label{poinc2}
 Let $\varepsilon>0$ and $N\in \N\smallsetminus\{0\}$. There exists a constant $C>1$ such that for any $p_{1}, p_{2}\in X/K$ with $r_{X}(p_{2})>\varepsilon$ and $N_{X}(p_{2})\leq N$, and any  $z_{1}, z_{2}\in G/K$, $x_{1}\in X$ such that $p_{1}=\Lambda z_{1}= x_{1}K$, $p_{2}=  \Lambda z_{2}$, one has 
\begin{align*}
 \int_{K}G(x_{1}k, B(p_{2}, \varepsilon))d\lambda_{K}(k)= C^{\pm 1}\mathfrak{p}(z_{1}, z_{2}, \delta_{G})
\end{align*}
\end{lemme}

\begin{proof}
The assumption that $r_{X}(p_{2})>\varepsilon$ means that the preimage in $G/K$ of $B(p_{2}, \varepsilon)$ is the disjoint union $\bigcup_{g\in \Lambda}B(gz_{2}, \varepsilon)$ where each ball appears with multiplicity $N_{X}(p_{2})$. Hence, given $g_{1}\in G$ such that $x_{1}=\Lambda g_{1}$, we can write
\begin{align*}
\int_{K}G(x_{1}k, B(p_{2}, \varepsilon))d\lambda_{K}(k) &= \frac{1}{N_{X}(p_{2})}\sum_{g\in \Lambda}\int_{K}G(g_{1}k, B(gz_{2}, \varepsilon))d\lambda_{K}(k) \\
&= c^{\pm 1} \frac{1}{N_{X}(p_{2})}\sum_{g\in \Lambda} e^{\delta_{G}d(z_{1}, gz_{2})} \\
&=c^{\pm 1} \frac{1}{N_{X}(p_{2})}\mathfrak{p}(z_{1}, z_{2}, \delta_{G})
\end{align*}
if $c>1$ is chosen as in \Cref{poinc1}.
\end{proof}

\bigskip

It is now easy to conclude that Fact 2 is a reformulation of \Cref{Kaimth} : 

\begin{proof}[Fact 2  $\iff$ \Cref{Kaimth}]
We   check that the Poincaré series $\mathfrak{p}(z_{1}, z_{2}, \delta_{G})$ is finite if and only if  the Green functions of the geodesic flow are locally integrable.  Let $z_{2}\in G/K$, $p_{2}=\Lambda z_{2}\in X/K$ its projection on $X/K$, $(\varepsilon,N)\in \R_{>0}\times \N $ such that  $r_{X}(p_{2})> \varepsilon $ and  $N_{X}(p_{2})\leq N $, and write $C>1$ the associated constant of   \Cref{poinc2}.  Let $E\subseteq X$ be a right $K$-invariant compact subset. \Cref{poinc2}, together with the $K$-invariance of $E$ and of the Haar measure $\lambda$ on $X$, implies that
 \begin{align}\label{eq22}
 \int_{E}G(x, B(p_{2}, \varepsilon)) d\lambda(x) &= \int_{E} \int_{K}G(xk, B(p_{2}, \varepsilon)) d\lambda_{K}(k) d\lambda(x)\nonumber\\
 &= \int_{E} C^{\pm 1} \mathfrak{p}(z, z_{2}, \delta_{G})  d\lambda(x) \nonumber\\
 &=  (Ce^{\delta_{G}R})^{\pm 1} \lambda(E)\mathfrak{p}(z_{1}, z_{2}, \delta_{G})  
 \end{align}
where $z\in G/K$ is any lift of $x$ (i.e. satisfies $\Lambda z=xK$), $z_{1}$ is the lift of some fixed arbitrary point in $E$, and $R$ is the diameter of the projection of $E$ in $X/K$. 

The equation (\ref{eq22}) implies that $ G(., B(p_{2}, \varepsilon))$ is locally  integrable if and only if the poincaré series  $\mathfrak{p}(z_{1}, z_{2}, \delta_{G})$ is finite. Notice this is also true if one replaces $B(p_{2}, \varepsilon)$ by any relatively compact subset $F\subseteq X$ with positive measure. More precisely, if $\mathfrak{p}(z_{1}, z_{2}, \delta_{G})<\infty$ then, covering $F$ by a finite number of balls $B(p_{2}, \varepsilon)$ with $\varepsilon<r_{X}(p_{2})$, we infer from above that $ G(., F)$ is locally integrable. Conversely, if  $\mathfrak{p}(z_{1}, z_{2}, \delta_{G})=+\infty$, then by \Cref{Kaimth}, the geodesic flow is recurrent ergodic, and as $F$ has positive measure, we necessarily have $G(.,F)=+\infty$ almost everywhere.

\end{proof}

\bigskip

\bibliographystyle{abbrv}

\bibliography{bibliographie}

\end{document}